\documentclass[a4paper,10pt]{article}
\usepackage[utf8]{inputenc}

\usepackage{comment,xcomment,amssymb,amsmath,color}
\usepackage{rotating,xypic,array}
\usepackage{latexsym,mathtools,amsthm}
\usepackage{subfig,rotating}
\usepackage{tikz,fancyvrb}
\usepackage{booktabs}  
\usetikzlibrary{arrows.meta}
\usepackage{enumitem}
\usepackage{url} 
\usepackage{longtable}
\usepackage[section]{placeins}
\usepackage[longtable]{multirow}
\usepackage[colorlinks,citecolor=blue]{hyperref}

\usepackage{pdflscape}

\makeatletter
\def\namedlabel#1#2{\begingroup
   \def\@currentlabel{#2}%
   \label{#1}\endgroup
}
\makeatother
\usepackage{geometry}

\title{Nice pseudo-Riemannian nilsolitons\\
\small{or, how to do math with a hammer}}
\author{Diego Conti and Federico A. Rossi}

\newtheorem{theorem}{Theorem}[section]
\newtheorem{lemma}[theorem]{Lemma}

\newtheorem{corollary}[theorem]{Corollary}
\newtheorem{proposition}[theorem]{Proposition}
\theoremstyle{definition}

\newtheorem{example}[theorem]{Example}
\theoremstyle{remark}
\newtheorem{remark}[theorem]{Remark}

\newcommand{\abs}[1]{\left\vert#1\right\vert}
\newcommand{\R}{\mathbb{R}}
\newcommand{\im}{\mathrm{Im}\,}         
\newcommand{\lie}[1]{\mathfrak{#1}}     
\newcommand{\g}{\lie{g}}
\newcommand{\Z}{\mathbb{Z}}

\newcommand{\C}{\mathbb{C}}

\newcommand{\hook}{\lrcorner\,}

\newcommand{\id}{\operatorname{Id}}   

\newcommand{\Span}[1]{\operatorname{Span}\left\{#1\right\}}
\newcommand{\tran}[1]{\hspace{.2mm}\prescript{t\hspace{-.5mm}}{}{#1}}
\DeclareMathOperator{\Ric}{Ric} 

\DeclareMathOperator{\Der}{Der}
\DeclareMathOperator{\ad}{ad}
\DeclareMathOperator{\coker}{coker}
\DeclareMathOperator{\logsign}{logsign}

\DeclareMathOperator{\Tr}{tr}
\DeclareMathOperator{\sign}{sign}

\newcolumntype{C}{>{$}c<{$}}
\newcolumntype{L}{>{$}l<{$}}
\newcolumntype{R}{>{$}r<{$}}

\newcommand{\tableottocoker}{A.1}
\newcommand{\tablenovecoker}{A.2}
\newcommand{\tableottoriemannian}{A.3}
\newcommand{\tablenoveriemannian}{A.4}
\begin{document}
\VerbatimFootnotes
\maketitle
\begin{abstract}
We study nice nilpotent Lie algebras admitting a diagonal nilsoliton metric. We classify nice Riemannian nilsolitons
up to dimension $9$. For general signature, we show that determining whether a nilpotent nice Lie algebra admits a nilsoliton metric reduces to a linear problem together with a system of as many polynomial equations as the corank of the root matrix. We classify nice nilsolitons of any signature: in dimension $\leq 7$; in dimension $8$ for corank $\leq 1$; in dimension $9$ for corank zero.
\end{abstract}

\renewcommand{\thefootnote}{\fnsymbol{footnote}}
\footnotetext{\emph{MSC class 2020}: \emph{Primary} 53C25; \emph{Secondary} 53C50, 53C30, 22E25}
\footnotetext{\emph{Keywords}: Einstein metrics, nilsoliton, nice Lie algebras, pseudo-Riemannian homogeneous metrics.}
\renewcommand{\thefootnote}{\arabic{footnote}}

\section*{Introduction}
A left-invariant metric on a nilpotent Lie group, or equivalently a metric on its Lie algebra, is called a \emph{nilsoliton} if
\begin{equation}
\label{eqn:nilsoliton}
\Ric=\lambda\id + D, \quad D\in\Der \g.
\end{equation}
The terminology is motivated by the fact that solutions of~\eqref{eqn:nilsoliton} are solitons for the Ricci flow (see~\cite{Lauret:RicciSoliton}). In the Riemannian setting, nilsolitons are studied because of their relation to Einstein solvmanifolds (\cite{Heber:noncompact,Lauret:Einstein_solvmanifolds}). In the pseudo-Riemannian case, a similar relation exists, but it is more complicated, depending on whether the scalar curvature is zero and the solvable Lie algebra is unimodular (see~\cite{ContiRossi:IndefiniteNilsolitons}). We are interested in the case that the Einstein solvable Lie algebra is nonunimodular and the scalar curvature is nonzero; this is the case that most resembles the Riemannian situation, although one has to put restrictions on the geometry in order to obtain a precise correspondence.

Given a Lie algebra $\tilde \g$ with a fixed metric, a pseudo-Iwasawa decomposition is an orthogonal splitting $\tilde\g=\g\oplus^\perp\lie a$, with $\g$ a nilpotent ideal and $\lie a$ an abelian subalgebra acting on $\g$ by self-adjoint derivations. If $\tilde \g$ is further assumed to be nonunimodular and with an Einstein metric of nonzero scalar curvature, then $\g$ is the nilradical and the induced metric satisfies~\eqref{eqn:nilsoliton} with $\lambda$ and $D$ nonzero; in the language of~\cite{ContiRossi:IndefiniteNilsolitons}, such a nilsoliton will be said to be of type Nil4. Conversely, every nilsoliton of type Nil4 can be extended to an Einstein, nonunimodular solvable Lie algebra of nonzero scalar curvature with a pseudo-Iwasawa decomposition (see~\cite{ContiRossi:IndefiniteNilsolitons}). In particular, one can choose an extension of the form $\g\rtimes_D\R$ (see also~\cite{Yan:Pseudo-RiemannianEinsteinhomogeneous}).

The goal of this paper is producing a large number of nilsolitons of any signature (including Riemannian). We consider nilpotent Lie algebras admitting a nice basis $\{e_i\}$: if $\{e^i\}$ is the dual basis, this means that each $[e_i,e_j]$ and each $e_i\hook de^j$ is a multiple of an element of the basis. Nice Lie algebras were introduced in~\cite{LauretWill:Einstein}, with the main motivation that the Ricci tensor of a metric which is diagonal relative to a nice basis is also diagonal. Riemannian nice nilsolitons are classified up to dimension $7$ (see \cite{Lauret:Finding,Will:RankOne,FernandezCulma}); some special classes of nilsolitons of dimension $8$ are classified in \cite{KadiogluPayne,Arroyo:Filiform}.

We point out that for diagonal metrics on a nice Lie algebra, the Ricci tensor is always diagonalizable; accordingly, all nilsoliton metrics in this paper have $D$ diagonalizable, though this is not always the case for indefinite nilsolitons (see~\cite{ContiRossi:IndefiniteNilsolitons}).

Nilpotent nice Lie algebras are classified up to dimension $9$ (see \cite{FernandezCulma,LauretWill:OnTheDiagonalization,ContiRossi:Construction}). Much information on the structure of a nice Lie algebra is encoded in its root matrix $M_\Delta$. In particular, for a fixed nice Lie algebra, one can consider deformations of the Lie bracket obtained by varying the nonzero structure constants relative to the nice basis, and the corank of $M_\Delta$ (i.e. the dimension of its cokernel) governs the dimension of this space of deformations taken up to diagonal rescalings and without taking into account the Jacobi equality (see \cite[Proposition~2.2]{ContiRossi:Construction}).

It was shown in~\cite{Nikolayevsky} that the problem of determining whether a nice Lie algebra admits a (diagonal) Riemannian nilsoliton metric can be reduced to solving a set of linear equalities and inequalities. This follows from the convexity of the scalar curvature functional, which does not hold for general pseudo-Riemannian metrics. However, we show that nice Lie algebra admitting diagonal nilsoliton metrics can be identified by solving a set of linear equalities and inequalities, together with as many polynomial equations as the corank of $M_\Delta$ (Corollary~\ref{cor:KHLP}). Applying this result to the list of nice nilpotent Lie algebras constructed in~\cite{ContiRossi:Construction}, we classify:
\begin{itemize}
\item Riemannian diagonal nice nilsolitons of dimension $\leq 9$;
\item pseudo-Riemannian diagonal nice nilsolitons of dimension $\leq 7$;
\item pseudo-Riemannian diagonal nice nilsolitons of dimension $8$ such that $\dim\coker M_\Delta\leq 1$;
\item pseudo-Riemannian diagonal nice nilsolitons of dimension $9$ such that $\dim\coker M_\Delta=0$.
\end{itemize}
We note that Riemannian nilsoliton metrics can be turned into indefinite nilsolitons by reversing the sign of the metric on a subspace spanned by appropriately chosen nice basis elements
(see Proposition~\ref{prop:wick}); a version of this observation was already considered in~\cite{Yan:Pseudo-RiemannianEinsteinhomogeneous}. For each nice nilpotent Lie algebra of dimension $\leq 9$ not included in our low-corank classifications mentioned above, we list the signatures of the nilsoliton metrics obtained in this way.

We point out that the nilsoliton metrics obtained in this paper determine Einstein nice solvable Lie algebras in one dimension higher (see Proposition~\ref{prop:NicePseudoIwasawa}).

\medskip
\noindent \textbf{Acknowledgments:} The authors acknowledge GNSAGA of INdAM. F.A.~Rossi also acknowledges the Young Talents Award of Universit\`{a} degli Studi di Milano-Bicocca joint with Accademia Nazionale dei Lincei.

\section{Nice pseudo-Riemannian nilsolitons}\label{sec:NicePseudoIwasawa}
In this section we study the nilsoliton equation for diagonal pseudo-Riemannian metrics on nice nilpotent Lie algebras. We give a characterization of nice nilpotent Lie algebras admitting a diagonal, non-Einstein nilsoliton metric, preparing for the classifications in later sections.

Recall from~\cite{ContiRossi:IndefiniteNilsolitons} that a metric $\langle,\rangle$ on a nilpotent Lie algebra is called a \emph{nilsoliton of type Nil4} if
\[\Ric=\lambda \id +D, \quad D\in\Der\g,\ \lambda\neq0,\ D\neq 0;\]
since we are considering diagonal metrics on nice Lie algebras, for which $\Ric$ is diagonalizable, all non-Einstein nilsolitons are of this type (see \cite[Theorem~2.1]{ContiRossi:IndefiniteNilsolitons}).

We begin by stating some results on nice nilpotent Lie algebras. A \emph{nice Lie algebra} is a pair $(\g,\mathcal{B})$ where $\g$ is a Lie algebra and $\mathcal{B}=\{e_1,\dotsc, e_n\}$ is a basis of $\g$ such that each bracket $[e_i,e_j]$ is a multiple of some $e_k$ and each interior product $e_i\hook de^j$ is a multiple of some $e^k$, having denoted by $e^1,\dotsc, e^n$ the dual basis. To a nice Lie algebra, we can associate a directed graph $\Delta$ with arrows labeled by nodes, called its \emph{nice diagram}: the nodes of $\Delta$ are the elements of the nice basis, and $e_i\xrightarrow{e_j}e_k$ is an arrow if $[e_i,e_j]$ is a nonzero multiple of $e_k$. To a nice diagram we associate the \emph{root matrix} $M_\Delta$, which has a row for every $(\{i,j\},k)$ such that $[e_i,e_j]$ is a nonzero multiple of $e_k$; the row associated to $(\{i,j\},k)$ has $+1$ in position $k$, $-1$ in positions $i$ and $j$, and zeroes in the other entries. The structure of the Lie algebra $\g$ is then completely determined by $M_\Delta$ and the collection of structure constants $\{c_{ij}^k\}$, where
\[[e_i,e_j]=c_{ij}^ke_k, \quad i<j.\]
We will represent the structure constants by a vector $c=(c_{ij}^k)$ by fixing the same ordering used to list the rows of $M_\Delta$. We will denote by $M_{\Delta,2}$ the $\operatorname{mod} 2$ reduction of $M_{\Delta}$.

Given a vector $v=(a_1,\dotsc, a_n)\in\R^n$, we denote by $v^D$ the diagonal matrix with elements $a_1,\dotsc, a_n$ on the diagonal. By \cite[Lemma~2]{Payne:TheExistence}, the space of diagonal derivations of a nice Lie algebra $(\g,\mathcal{B})$
is given by
\[\{v^D\mid v\in\ker M_\Delta\}.\]
We denote by $[1]_k$, or simply $[1]$ when $k$ is implied by the context, the vector in $\R^k$ with all entries equal to $1$.

Recall that a derivation $N$ on a Lie algebra $\g$ is called a \emph{Nikolayevsky} (or \emph{pre-Einstein}) \emph{derivation} if it is semisimple and
\begin{equation}
 \label{eqn:nik}
 \Tr(NX)=\Tr X, \quad X\in\Der\g.
\end{equation}
By~\cite{Nikolayevsky}, every Lie algebra admits a unique Nikolaevsky derivation up to automorphisms. On a nice Lie algebra, we can fix a canonical Nikolayevsky derivation:
\begin{proposition}[{\cite[Theorem~3.1]{Payne:Applications}}]
\label{prop:payne}
Let
$b$ be a solution to
\[M_\Delta\tran{M_\Delta} b =[1].\]
and let $v=\tran{M_\Delta} b+[1]$. Then $N=v^D$ is a Nikolayevsky derivation.
\end{proposition}
We will refer to the derivation $v^D$ of Proposition~\ref{prop:payne} as the \emph{diagonal Nikolayevsky derivation}.

\begin{remark}
If the eigenvalues of the Nikolayevsky derivation are distinct, the Lie algebra is nice
(\cite{Nikolayevsky:EinsteinDerivation}).
\end{remark}

A diagonal metric on a nice Lie algebra of dimension $n$ takes the form
\[g_1e^1\otimes e^1+\dots + g_ne^n\otimes e^n;\]
accordingly, we will identify the metric with the vector $(g_1,\dotsc, g_n)\in\R^n$. We will say that the \emph{signature} of the metric is the vector $\delta=(\delta_1,\dotsc, \delta_n)$ in $(\Z_2)^n$, where $\delta_i$ is zero or one accordingly to whether $g_i$ is positive or negative; notice that $\delta$ determines the signature in the usual sense. We also write $\logsign g_i=\delta_i$, the notation motivated by the relation $(-1)^{\delta_i}=\sign{g_i}$.

The linear map $M_\Delta\colon\R^n\to\R^m$ can be viewed as a homomorphism of abelian Lie algebras, which exponentiates to a Lie group homomorphism
\[e^{M_\Delta}\colon (\R^n)^*\to(\R^m)^*,\]
where $(\R^n)^*$ is identified with the group of invertible diagonal matrices of order $n$. Explicitly, if the $h$-th row of $M_\Delta$ is $\tran(-e_i-e_j+e_k)$, the $h$-th component of
$e^{M_\Delta}(g)$ is $\dfrac{g_k}{g_ig_j}$.

We then have:
\begin{proposition}[\cite{ContiRossi:EinsteinNice}]
\label{prop:ricci}
Let $g$ be a diagonal metric on a nice Lie algebra with diagram $\Delta$ and structure constants $c$. Define $X$ by
\[X^D= e^{M_\Delta}(g)(c^D)^2.\]
Then the Ricci operator is given by
\[\Ric = \frac12 \tran{M_\Delta} X.\]
\end{proposition}

\begin{remark}
A diagonal metric on a direct sum of two nice Lie algebras is a nilsoliton if and only if it is the orthogonal sum of two nilsoliton metrics with the same constant $\lambda$. This is a simple consequence of the fact that the diagonal Nikolayevsky derivation restricts to the diagonal Nikolayevsky derivations on each component. For this reason, our classification results will be concerned only with irreducible nice Lie algebras in the sense of~\cite{ContiRossi:Construction}, meaning that it is not possible to write the nice basis as a disjoint union $\mathcal{B}'\cup\mathcal{B}''$, with each of $\mathcal{B}'$, $\mathcal{B}''$ spanning an ideal.
\end{remark}

\begin{theorem}
\label{thm:Xequalsb}
Let $\g$ be a nice nilpotent Lie algebra and let $b$ be a solution to $M_\Delta\tran{M_\Delta} b =[1]$. Given a diagonal pseudo-Riemannian metric $g$, the following are equivalent:
\begin{enumerate}
\item $g$ is a diagonal nilsoliton with $\Ric=\lambda \id+D$;
\item $\Ric = \lambda(\id-v^D)$, where $v^D$ is the diagonal Nikolayevsky derivation;
\item $X\in -2\lambda b + \ker\tran{M_\Delta}$, where $X^D= e^{M_\Delta}(g)(c^D)^2.$
\end{enumerate}
\end{theorem}
\begin{proof}
We first prove $1\implies 3$. By Proposition~\ref{prop:ricci}, the nilsoliton condition amounts to
\[\frac12 (\tran{M_\Delta} X)^D =\Ric= \lambda \id + D;\]
this implies that $D$ is a diagonal matrix, so it is equivalent to
\[\frac12 (\tran{M_\Delta} X) -\lambda[1] \in \ker M_\Delta,\]
i.e.
\[0=\frac12 M_\Delta\tran{M_\Delta} X +\lambda[1]=\frac12 M_\Delta\tran{M_\Delta} (X+2\lambda b).\]
This is equivalent to $Y=X+2\lambda b$ being in the kernel of $\tran{M_\Delta}$, as
\[\langle \tran{M_\Delta} Y ,\tran{M_\Delta}Y\rangle = \langle M_\Delta\tran{M_\Delta} Y ,Y\rangle = 0.\]
$3\implies 2$ follows from
Proposition~\ref{prop:ricci} and Proposition~\ref{prop:payne}, as
\[\Ric = \frac12(\tran{M_\Delta}(-2\lambda b))^D=-\lambda(v-[1])^D=\lambda (\id - v^D).\]
$2\implies1$ is obvious.
\end{proof}

Given a multi-index $\alpha=(a_1,\dotsc, a_m)\in\Z^m$ and a vector $X=(x_1,\dotsc, x_m)\in\R^m$, define
\[\abs{X}^\alpha=(\abs{x_1}^{a_1},\dotsc, \abs{x_m}^{a_m}), \quad
X^\alpha=(x_1^{a_1},\dotsc, x_m^{a_m}).\]

\begin{corollary}
\label{cor:KHLP}
Let $\g$ be a nice nilpotent Lie algebra and let $b$ be a solution to $M_\Delta\tran M_\Delta b=[1]$. There exists a diagonal nilsoliton metric of signature $\delta$ if and only if there is a vector $X\in\R^m$ such that
\begin{description}
\item[($\mathbf{K}$)\namedlabel{enum:condK}{($\mathbf{K}$)}] $X\in-2\lambda b + \ker\tran M_\Delta$, where $\lambda$ is the constant appearing in the equation $\Ric=\lambda \id+D$;
\item[($\mathbf{H}$)\namedlabel{enum:condH}{($\mathbf{H}$)}] $X$ does not belong to any coordinate hyperplane;
\item[($\mathbf{L}$)\namedlabel{enum:condL}{($\mathbf{L}$)}] $\logsign X=M_{\Delta,2}\delta$;
\item[($\mathbf{P}$)\namedlabel{enum:condP}{($\mathbf{P}$)}] for some (hence every) choice of $\alpha_1,\dotsc, \alpha_k\in \Z^m$ forming a basis of $\ker \tran{M_\Delta}$, we have
\[X^{\alpha_i}=c^{2\alpha_i}, \quad i=1,\dotsc, k,\] where $c$ is the vector of structure constants of $\g$.
\end{description}
\end{corollary}
\begin{proof}
By Theorem~\ref{thm:Xequalsb}, we must show that $X=(x_1,\dotsc, x_m)$ satisfies \ref{enum:condH}, \ref{enum:condK}, \ref{enum:condL}, \ref{enum:condP} if and only if
\begin{equation}
 \label{eqn:cond3riscritta}
X\in-2\lambda b+\ker\tran M_\Delta, \qquad X^D=e^{M_\Delta}(g)(c^D)^2,
\end{equation}
where $g$ has signature $\delta$. We argue as in \cite[Theorem~2.2]{ContiRossi:RicciFlat}. Assume~\eqref{eqn:cond3riscritta} holds with $g=(-1)^\delta \exp v$. Then
\[\logsign X=M_{\Delta,2}\delta, \qquad \abs{X}^D=\exp (M_\Delta(v)^D)(c^D)^2,\]
and \ref{enum:condK}, \ref{enum:condH}, \ref{enum:condL} hold trivially. In addition, the vector
\[\log \abs{X}-2\log\abs{c}=\begin{pmatrix}\log \abs{x_1}-2\log \abs{c_1}\\
                             \vdots\\
                             \log \abs{x_m}-2\log \abs{c_m}\\
                            \end{pmatrix}\]
lies in the image of $M_\Delta$; this is equivalent  to $\alpha_i(\log\abs{X}-\log c^2)=0$ for $i=1,\dotsc, k$, i.e.
\[\abs{X}^{\alpha_i}=c^{2\alpha_i}.\]
Now write $\alpha_{i,2}$ for the $\operatorname{mod} 2$ reduction of $\alpha_i$, and observe that
\[\logsign X^{\alpha_i}=\langle \logsign X,\alpha_{i,2}\rangle = \langle M_{\Delta,2}v, \alpha_{i,2}\rangle =\langle v, \tran M_{\Delta,2}\alpha_{i,2}\rangle=0,\]
which implies~\ref{enum:condP}.

The converse is proved in a similar way.
\end{proof}

\begin{remark}
Corollary~\ref{cor:KHLP} is a generalization of \cite[Theorem~2.2]{ContiRossi:RicciFlat}, where either $\lambda=0$ or $N=0$ was assumed. Even in those cases, the result of Corollary~\ref{cor:KHLP} is stronger because $X$ appears instead of $\abs{X}$ in~\ref{enum:condP}.

In this paper, we are interested in diagonal nilsoliton metrics of type Nil4 on nice Lie algebras, which only exist if the Nikolaevsky derivation $N$ is nonzero (see Theorem~\ref{thm:Xequalsb}).
\end{remark}

\begin{remark}
Eliminating denominators, \ref{enum:condP} determines a system of homogeneous polynomial equations in the entries of $X$. Indeed, every $\alpha=(a_1,\dotsc, a_m)$ in $\ker \tran M_\Delta$ is orthogonal to $\im M_\Delta$, and in particular to $M_\Delta[1]=-[1]$; therefore, we have $a_1+\dotsc + a_m=0$.
\end{remark}

\begin{remark}
 \label{rem:wick}
Given a nice Lie algebra and a diagonal nilsoliton metric $g$, Corollary~\ref{cor:KHLP} implies that $(-1)^\delta g$ is also a nilsoliton for any $\delta\in\ker M_{\Delta,2}$. This can also be viewed as a consequence of Proposition~\ref{prop:ricci}: $g$ and $(-1)^\delta g$ induce the same vector $X$, so they have the same Ricci operator.
\end{remark}

\begin{example}
\label{example:lorentziano}
Consider the nice Lie algebra
\[\texttt{7421:9} \qquad (0,0,0,-e^{12},e^{13},e^{14}+e^{23},e^{16}+e^{34}).\]
This notation, which will be used throughout the paper, means that relative to a nice basis $\{e_1,\dotsc, e_7\}$ and its dual $\{e^1,\dotsc, e^7\}$, the Chevalley-Eilenberg differential satisfies
\[de^1=0, de^2=0, de^3=0, de^4=-e^1\wedge e^2, de^5=e^1\wedge e^3, de^6=e^1\wedge e^4+e^2\wedge e^3, de^7=e^1\wedge e^6+e^3\wedge e^4;\]
the label \texttt{7421:9} refers to the classification of~\cite{ContiRossi:Construction}.

Accordingly, the root matrix is given by
\[\begin{pmatrix}
-1 & -1 & 0 & 1 & 0 & 0 & 0\\
-1 & 0 & -1 & 0 & 1 & 0 & 0\\
-1 & 0 & 0 & -1 & 0 & 1 & 0\\
0 & -1 & -1 & 0 & 0 & 1 & 0\\
-1 & 0 & 0 & 0 & 0 & -1 & 1\\
0 & 0 & -1 & -1 & 0 & 0 & 1
\end{pmatrix};\]
the Nikolayevsky derivation is
\[N=\frac2{19}(3,5,6,8,9,11,14)^D.\]
Fixing $\lambda=-\frac12$, the vectors $X$ satisfying \ref{enum:condK} are given by
\begin{equation}
\label{eqn:X7421:9}
X=\bigl( \frac{1}{19}+x_6,-\frac{1}{19},\frac{4}{19},\frac{8}{19}-x_6,\frac{9}{19}-x_6,x_6\bigr),
\end{equation}
where $x_6$ is a parameter. Notice that \ref{enum:condH} holds for $x_6$ different from $-\frac1{19},\frac8{19},\frac9{19}$. Since $X$ is uniquely determined up to $\ker \tran M_\Delta$, it is evident from~\eqref{eqn:X7421:9} that $\ker \tran M_\Delta$ is spanned by
\[\alpha=(1,0,0,0,-1,-1,1).\]
Therefore, $X$ satisfies \ref{enum:condP} when
\[\frac{(\frac{1}{19}+x_6)x_6}{(\frac{8}{19}-x_6)(\frac{9}{19}-x_6)}=1,\]
i.e. for $x_6=\frac4{19}$, giving $X=\frac1{19}(5,-1,4,4,5,4)$ and $\logsign X=(0,1,0,0,0,0)$. We see that $M_{\Delta,2}\delta=\logsign X$ has the four solutions
\[(0,0,0,0,1,0,0),\; (1,1,0,0,0,1,0),\; (1,0,0,1,0,0,1),\; (0,1,0,1,1,1,1),\]
corresponding to four different signatures $\delta\in(\Z_2)^7$ of diagonal nilsoliton metrics. In particular, this Lie algebra admits a Lorentzian nilsoliton metric with $g_5<0$, but not a Riemannian nilsoliton metric, consistently with~\cite{FernandezCulma}.

In order to compute the metrics explicitly, we can argue as in the proof of Corollary~\ref{cor:KHLP} and write $X=e^{M_\Delta}(g)$ with $g=(-1)^\delta \exp v$. The linear system $\log \abs{X}=M_\Delta(v)$ has the solution
\begin{equation}
\label{eqn:v7421:9}
v=\left(0,0,\log \frac5{19}, \log \frac5{19},\log \frac5{361},\log \frac{20}{361},\log \frac{100}{6859}\right).
\end{equation}
Therefore, a Lorentzian nilsoliton metric with $\Ric=-\frac12(\id-N)$ is given by
\[e^1\otimes e^1 + e^2\otimes e^2 + \frac5{19}(e^3\otimes e^3+e^4\otimes e^4)-\frac5{361}e^5\otimes e^5+\frac{20}{361}e^6\otimes e^6+\frac{100}{6859}e^7\otimes e^7,\]
and the other signatures are obtained by changing the signs appropriately.

Note that the solution~\eqref{eqn:v7421:9} is not unique, but adding an element of $\ker\tran M_\Delta$ to $v$ yields the same metric up to an isomorphism of nice Lie algebras (see \cite[Remark~2.3]{ContiRossi:RicciFlat}).
\end{example}

\begin{example}
An example where there is a vector $X$ satisfying \ref{enum:condK} and \ref{enum:condH}, but not satisfying \ref{enum:condL}, is the following. Take
\[\texttt{85421:4a}
\quad (0,0,0,e^{12},e^{14},e^{13}+e^{24},e^{15},e^{17}+e^{23}).\]
Then the only solution to \ref{enum:condK} is
\[X=\bigl(\frac{3}{22},\frac{5}{22},-\frac{1}{11},\frac{5}{22},\frac{7}{22},\frac{2}{11},\frac{2}{11}\bigr);\]
this vector is not in any coordinate hyperplane, but there is no $\delta$ such that $\logsign M_{\Delta,2}(\delta)=\logsign X$. The same holds for \texttt{85421:4b}, which only differs by a sign in front of $e^{14}$.

Examples where \ref{enum:condH} or \ref{enum:condP} fail appear in Theorem~\ref{thm:nice7}.
\end{example}

The metrics $g$ and $(-1)^\delta g$ of Remark~\ref{rem:wick} can be related geometrically by the following construction. Fix a Lie algebra $\g$ with a nice basis $\mathcal{B}=\{e_1,\dotsc, e_n\}$, and let $D=(d_1,\dotsc, d_n)\in\Z^n$. Inside the complexification of $\g$, consider the subspace
\begin{equation}
\label{eqn:wickbasis}
\g_D=\operatorname{Span}_{\R}\mathcal{B}_D, \quad \mathcal{B}_D=\{i^{d_1}e_1,\dotsc, i^{d_n}e_n\}.
\end{equation}
If $\g_D$ is a real subalgebra of $\g^\C$, then $\g$ and $\g_D$ share the same complexification.
\begin{proposition}
\label{prop:gD}
Let $\g$ be a Lie algebra with a nice basis $\{e_1,\dotsc, e_n\}$; choose $D\in \Z^n$, and let $\delta$ be its $\operatorname{mod}2$ reduction. Then:
\begin{itemize}
\item if $M_{\Delta,2}(\delta)=0$, then $\g_D$ is a subalgebra of $\g^\C$ with a nice basis given by~\eqref{eqn:wickbasis}; in particular, $\g$ and $\g_D$ have the same root matrix;
\item if $\delta$ is the $\operatorname{mod} 2$ reduction of some $D'$ with $M_\Delta(D')=0$, then $\g_D$ is isomorphic to $\g$ as a nice Lie algebra;
\item conversely, if $\g_D$ is a subalgebra of $\g^\C$ then $M_{\Delta,2}(\delta)=0$.
\end{itemize}
\end{proposition}
\begin{proof}
Write $E_1=i^{d_1}e_1,\dotsc, E_n=i^{d_n}e_n$. Suppose $[e_h,e_k]$ is a nonzero multiple of $e_l$, say $[e_h,e_k]=c_{hk}^l e_l$. Then $[E_h,E_k]=c_{hk}^li^{d_h+d_k-d_l} E_l$; this is a real multiple of $E_l$ if and only if $d_h+d_k-d_l$ is even. Imposing this condition for all nonzero brackets is equivalent to $M_{\Delta,2}(\delta)=0$.

In the case that $M_{\Delta}(D)=0$, we have $d_h+d_k-d_l=0$, so $\g$ and $\g_D$ have the same structure constants.

On the other hand if $D$ and $D'$ have the same $\operatorname{mod} 2$ reduction $\delta$, $\g_D$ and $\g_{D'}$ coincide as a subset of $\g^\C$ and the nice bases $\mathcal{B}_D$ and $\mathcal{B}_{D'}$ only differ by changing signs; therefore, the two nice Lie algebras are isomorphic.
\end{proof}

\begin{remark}
Not every $\delta\in\ker M_{\Delta,2}$ can be lifted to an element of $\ker M_\Delta$ in general. This happens precisely when the nice diagram $\Delta$ gives rise to more than one (family of) nice Lie algebras, with different signs of the structure constants (see~\cite{ContiRossi:Construction}).
\end{remark}

Now let $\g$ be a nice Lie algebra with a diagonal metric $g$, relative to the nice basis $\mathcal{B}$, and take $D\in \Z^n$. On the nice Lie algebra $(\g_D,\mathcal{B}_D)$ we have an induced diagonal metric $g_D$, obtained by $\C$-bilinear extension to $\g^\C$ and then restriction. In the language of~\cite{Helleland:WickRotations}, $(\g_D,g_D)$ is a \emph{Wick rotation} of $(\g,g)$. On the other hand, one can also take the metric on $\g_D$ defined by pulling back the metric on $\g$ under the linear isomorphism mapping elements of $\mathcal B_D$ to $\mathcal B$ in the natural order; we will call this the \emph{transferred metric} on $\g_D$.

It follows easily from Corollary~\ref{cor:KHLP} that Wick rotation and transfer preserve the nilsoliton condition. This gives a more geometric interpretation of the observation of Remark~\ref{rem:wick}:
\begin{proposition}
\label{prop:wick}
Let $g$ be a diagonal metric on a nice Lie algebra $\g$. Then taking a Wick rotation $(\g_D,g_D)$ and transferring $g_D$ to $\g$ results in a diagonal metric $g'=(-1)^\delta g$, $\delta\in\ker M_{\Delta,2}$; in particular, $g$ is a nilsoliton if and only if $g'$ is a nilsoliton.

Conversely, any two metrics $g$, $(-1)^\delta g$ with $\delta\in\ker M_{\Delta,2}$ are related in this way.
\end{proposition}
\begin{proof}
If $D=(d_1,\dotsc, d_n)$, evaluating $g_D$ on the basis $\mathcal{B}_D$ gives
\[g_D(i^{d_k}e_k,i^{d_k}e_k)=(-1)^{d_k}g(e_k,e_k)=(-1)^{\delta_k}g(e_k,e_k),\]
where $\delta=(\delta_1,\dotsc, \delta_k)$ is the $\operatorname{mod}2$ reduction of $D$. In other words, if $g=(g_1,\dotsc, g_n)$, then $g_D=((-1)^{\delta_1}g_1,\dotsc, (-1)^{\delta_n}g_n)$. Transferring $g_D$ to $\g$ therefore gives $g'=(-1)^\delta g$. By Proposition~\ref{prop:gD}, $\delta$ is in $\ker M_{\Delta,2}$.

Conversely, suppose $g'=(-1)^\delta g$, with $M_{\Delta,2}(\delta)=0$, and let $D$ be a vector in $\Z^n$ whose $\operatorname{mod}2$ reduction is $\delta$. By the above calculation, transferring $g'$ to $\g_D$ gives the Wick-rotated metric $g_D$.
\end{proof}

Note that Example~\ref{example:lorentziano} shows that not all the pseudo-Riemannian nilsolitons are obtained from a Riemannian nilsoliton by transferring the metric.

\begin{example}
The $6$-dimensional Lie algebras
\begin{gather*}
\texttt{631:5a}\qquad (0, 0, 0, e^{12}, e^{13}, e^{24}+e^{35})\\
\texttt{631:5b}\qquad (0, 0, 0, -e^{12}, e^{13}, e^{24}+e^{35})
\end{gather*}
are not isomorphic over $\R$, but they share the same nice diagram $\Delta$ and the same complexification. We can proceed as in Example~\ref{example:lorentziano}, obtaining the following:
\begin{equation}
 \label{eqn:X631:5}
X=\left(\frac{1}{4},\frac{1}{4},\frac{1}{4},\frac{1}{4}\right),\quad
\delta\in\operatorname{Span}_{\Z_2}\{(1,0,0,1,1,1),(0,1,0,1,0,0),(0,0,1,0,1,0)\}.
 \end{equation}
The generic vector $v$ such that $\log\abs{X}=M_{\Delta}(v)$ is given by:
\[v=\left(v_1,v_2,v_2,v_1+v_2-\log4,v_1+v_2-\log 4,v_1+2 v_2-2 \log 4\right).\]
In the end, on both \texttt{631:5a} and \texttt{631:5b}, for any signature $\delta$ as in~\eqref{eqn:X631:5} we have a $2$-parameter family of nilsoliton metrics given by $g=(-1)^\delta \exp v$. For example, on $\g=\texttt{631:5a}$, for $\delta=(0,1,0,1,0,0)$ we have the nilsoliton metric
\[g=e^1\otimes e^1 - e^2\otimes e^2 +  e^3\otimes e^3 \\
- \frac{1}{4}e^4\otimes e^4 + \frac{1}{4} e^5\otimes e^5 +  \frac{1}{16} e^6\otimes e^6.\]
If we set $D=(1,1,0,0,1,1)\in\Z^6$, then $\g^D$ is isomorphic to \texttt{631:5b}, and the Wick-rotated metric $g^D$ is
\[-e^1\otimes e^1 + e^2\otimes e^2 +  e^3\otimes e^3 \\
- \frac{1}{4}e^4\otimes e^4 - \frac{1}{4} e^5\otimes e^5 -  \frac{1}{16} e^6\otimes e^6.\]
On the other hand, if we set $D=(0,1,0,1,0,0)$, then $\g^D$ is isomorphic to $\g$ itself and the Wick-rotated metric $g^D$ is Riemannian.
\end{example}

\begin{remark}
It is known that given two nilpotent Lie algebras with the same complexification, if one is a Riemannian nilsoliton then so is the other (\cite{Nikolayevsky,Jablonski:Orbits}. In our context, this is reflected in the fact that a diagonal nilsoliton metric on a nice Lie algebra $\g$ can be transferred to a nilsoliton metric on $\g_D$ with the same signature.
\end{remark}

For positive signatures $\delta$, a result of~\cite{Nikolayevsky} implies that condition~\ref{enum:condP} is redundant; in fact, this also holds slightly more generally:
\begin{corollary}
\label{cor:positivesignature}
Let $\g$ be a nice nilpotent Lie algebra and let $v^D$ be the diagonal Nikolayevsky derivation. If $\delta$ is a signature with $M_{\Delta,2}\delta=0$, then $\g$ has a diagonal nilsoliton metric of signature $\delta$ if and only if there is a vector $X\in\R^m$ with all entries positive satisfying  \ref{enum:condK}.
\end{corollary}
\begin{proof}
Given $X$ as in the statement, \cite[Lemma~2]{Nikolayevsky} implies that there exists a Riemannian nilsoliton metric. By the observation of  Remark~\ref{rem:wick}, a nilsoliton metric exists for any $\delta$ in $\ker M_{\Delta,2}$.
\end{proof}
\begin{remark}
Example~\ref{example:lorentziano} shows that not all nilsoliton metrics are obtained by transferring a Riemannian nilsoliton metric.
\end{remark}

Recall from~\cite{ContiRossi:IndefiniteNilsolitons} that, given a metric $\widetilde{\langle,\rangle}$ on a Lie algebra $\tilde\g$, a \emph{pseudo-Iwasawa decomposition} is an orthogonal splitting $\tilde\g=\g\oplus\lie a$, where $\lie g$ is an ideal, $\lie a$ is an abelian subalgebra, and $\ad X$ is self-adjoint for all $X\in\lie a$. There is a correspondence between nilsolitons and pseudo-Iwasawa Einstein solvmanifolds (see~\cite{ContiRossi:IndefiniteNilsolitons}), which can be specialized to the nice setting as follows:
\begin{proposition}
\label{prop:niceEinsteinNilradical}
Let $\tilde\g$ be a nonunimodular nice solvable Lie algebra with a diagonal metric of pseudo-Iwasawa type satisfying the Einstein equation $\widetilde{\Ric}=\lambda \id$, $\lambda\neq0$. Then $[\tilde\g,\tilde\g]$ with the induced metric is a nice Lie algebra with a diagonal nilsoliton metric of type Nil4 satisfying
\[\Ric=\lambda \id-\lambda N,\]
where $N$ is the diagonal Nikolayevsky derivation.
\end{proposition}
\begin{proof}
If $\{e_i\}$ is the nice basis, the derived Lie algebra is spanned by the $e_k$ such that some Lie bracket $[e_i,e_j]$ is a nonzero multiple of $e_k$; this shows that the derived algebra is nice and the induced metric is diagonal. By \cite[Corollary~3.13]{ContiRossi:IndefiniteNilsolitons}, the derived algebra is a nilsoliton with $\Ric=\lambda \id+D$, $D=\ad H$ and $\Tr D\neq0$. By Theorem~\ref{thm:Xequalsb}, we have $D=-\lambda N$.
\end{proof}

\begin{proposition}
\label{prop:NicePseudoIwasawa}
Let $\g$ be a nice nilpotent Lie algebra with a diagonal metric $\langle,\rangle$ of type Nil4. Then $\Ric=\lambda \id-\lambda N$, where $N$ is the (nonzero) diagonal Nikolayevsky derivation, and the semidirect product $\tilde\g=\g\rtimes_N\Span{e_0}$ is a nice solvable Lie algebra with an Einstein diagonal pseudo-Iwasawa metric
\[\langle,\rangle -\frac{\Tr N}{\lambda}e^0\otimes e^0.\]
\end{proposition}
\begin{proof}
Write $\Ric=\lambda \id+D$; by Theorem~\ref{thm:Xequalsb} we have $D=-\lambda N$. By \cite[Corollary~4.7]{ContiRossi:IndefiniteNilsolitons} (see also \cite[Theorem~4.7]{Yan:Pseudo-RiemannianEinsteinhomogeneous}), we have that $\g\rtimes_D\Span{E_0}$ has an Einstein pseudo-Iwasawa metric of the form
$\widetilde{\langle,\rangle}=\langle,\rangle+(\Tr D) E^0\otimes E^0$; setting $e_0=-\frac1\lambda E_0$, the metric takes the form
\[\widetilde{\langle,\rangle}=\langle,\rangle+\frac{\Tr D}{\lambda^2} e^0\otimes e^0=\langle,\rangle-\frac{\Tr N}{\lambda} e^0\otimes e^0.\]

Since $\g$ is nice and $N$ is diagonal, it is now obvious that $\tilde\g$ is nice and pseudo-Iwasawa.
\end{proof}

\begin{example}
Taking $\g$ as in Example~\ref{example:lorentziano} and applying Proposition~\ref{prop:NicePseudoIwasawa}, we obtain an Einstein metric on the solvable Lie algebra $\tilde\g=\g\rtimes_N\Span{e_0}$ of the form
\[e^1\otimes e^1 + e^2\otimes e^2 + \frac5{19}(e^3\otimes e^3+e^4\otimes e^4)-\frac5{361}e^5\otimes e^5+\frac{20}{361}e^6\otimes e^6+\frac{100}{6859}e^7\otimes e^7+\frac{224}{19}e^0\otimes e^0.\]
By construction, we have $\Ric=-\frac12\id$.
\end{example}

\section{Nilsolitons of dimension $\leq 7$}
It is well known that every nilpotent Lie algebra of dimension $\leq 6$ is a Riemannian nilsoliton (\cite{Will:RankOne,LauretWill:OnTheDiagonalization}). For dimension $7$, a classification is done in~\cite{FernandezCulma}. In this section we consider more generally nilsolitons of indefinite signature, though restricting to nice Lie algebras. A priori, this class of metrics includes both Einstein metrics and nilsolitons of type Nil4; since the Einstein case has been studied in \cite{ContiRossi:RicciFlat,ContiRossi:EinsteinNice}, we will focus on Nil4 metrics.

We will consider diagonal metrics relative to the nice basis, and normalize $\lambda$ to $-\frac12$, so that~\ref{enum:condK} reads
\begin{equation}\label{eq:condKnormalized}
M_\Delta \tran M_{\Delta} X= [1],
\end{equation}
and the nilsoliton equation of type Nil4
\begin{equation}
\label{eqn:normalizednil4}
\Ric=-\frac12\id +\frac12 N, \quad N\neq 0.
\end{equation}
Thus, we only consider Lie algebras with nonzero Nikolayevsky derivation; in particular, they are all graded.

To each nice Lie algebra $\g$ we associate the set
\[\mathbf{S}=\{\delta\mid \delta \text{ is the signature of a diagonal metric on $\g$ satisfying~\eqref{eqn:normalizednil4}}\}.\]
For fixed $\g$, the signature of a diagonal metric $g_1e^1\otimes e^1+\dotsc + g_ne^n\otimes e^n$ can be represented by the set of indices $i$ such that $g_i<0$. For instance, the signature of $-e^1\otimes e^1+e^2\otimes e^2-e^3\otimes e^3$ will be represented by the string $13$, and the signature of a positive definite metric by $\emptyset$. In addition, we will indicate by $\mathbf{S}_0$ the intersection of $\mathbf{S}$ with $\ker M_{\Delta,2}$, corresponding to signatures of nilsoliton metrics obtained by applying Corollary~\ref{cor:positivesignature}.

\begin{theorem}
\label{thm:nice6}
For each nice nilpotent Lie algebra of dimension $\leq 6$, the signatures of the diagonal nilsoliton metrics satisfying~\eqref{eqn:normalizednil4} are given in Table~\ref{table:Nilsoliton6}.
\end{theorem}
\begin{proof}
The proof is a case-by-case calculation using Corollary~\ref{cor:KHLP}, as in Example~\ref{example:lorentziano}. We used an updated version of the computer program~\cite{demonblast}, initially developed to obtain the classifications of~\cite{ContiRossi:Construction,ContiRossi:RicciFlat}. For all cases considered here, Condition~\ref{enum:condP} is either trivial or it boils down to an equation of degree $\leq 2$ in one variable, which the program solves automatically.
\end{proof}

\begin{footnotesize}
{\setlength{\tabcolsep}{2pt}
\begin{longtable}[htpc]{>{\ttfamily}c C C C}
\caption{Irreducible nice nilpotent Lie algebras of dimension $\leq 6$ that admit a diagonal nilsoliton metric of type Nil4\label{table:Nilsoliton6}}\\
\toprule
\textnormal{Name} & \g &   N & \mathbf{S} \\
\midrule
\endfirsthead
\multicolumn{4}{c}{\tablename\ \thetable\ -- \textit{Continued from previous page}} \\
\toprule
\textnormal{Name} & \g &   N & \mathbf{S} \\
\midrule
\endhead
\bottomrule\\[-7pt]
\multicolumn{4}{c}{\tablename\ \thetable\ -- \textit{Continued to next page}} \\
\endfoot
\bottomrule\\[-7pt]
\endlastfoot
31:1&0,0,e^{12}&2/3(1,1,2)&\{\emptyset,12,13,23\}\\[2pt]
421:1&0,0,e^{12},e^{13}&1/3(1,2,3,4)&\{\emptyset,124,13,234\}\\[2pt]
5321:1&0,0,e^{12},e^{13},e^{14}&1/12(2,9,11,13,15)&\{\emptyset,124,135,2345\}\\
5321:2&0,0,e^{12},e^{13},e^{14}+e^{23}&3/11(1,2,3,4,5)&\{\emptyset,135\}\\
532:1&0,0,e^{12},e^{13},e^{23}&5/12(1,1,2,3,3)&\{\emptyset,1245,135,234\}\\
521:2&0,0,0,e^{12},e^{24}+e^{13}&1/7(4,3,6,7,10)&\{\emptyset,125,145,24\}\\
52:1&0,0,0,e^{12},e^{13}&1/4(2,3,3,5,5)&\{\emptyset,123,125,134,145,2345,24,35\}\\
51:2&0,0,0,0,e^{12}+e^{34}&3/4(1,1,1,1,2)&\{\emptyset,12,1234,135,145,235,245,34\}\\[2pt]
64321:1&0,0,e^{12},e^{13},e^{14},e^{15}&1/11(1,9,10,11,12,13)&\{\emptyset,1246,135,23456\}\\
64321:2&0,0,e^{12},e^{13},e^{14},e^{15}+e^{23}&13/68(1,3,4,5,6,7)&\{\emptyset,1246\}\\
64321:3&0,0,- e^{12},e^{13},e^{14},e^{34}+e^{25}&1/6(1,3,4,5,6,9)&\{\emptyset,1246,1356,2345\}\\
64321:4&0,0,e^{12},e^{13},e^{23}+e^{14},e^{24}+e^{15}&3/13(1,2,3,4,5,6)&\{\emptyset,135\}\\
64321:5&0,0,- e^{12},e^{13},e^{23}+e^{14},e^{34}+e^{25}&11/52(1,2,3,4,5,7)&\{\emptyset,1356\}\\
6431:1&0,0,e^{12},e^{13},e^{23},e^{14}&1/4(1,2,3,4,5,5)&\{\emptyset,1245,1356,2346\}\\
6431:2a&0,0,e^{12},e^{13},e^{23},e^{14}+e^{25}&7/20(1,1,2,3,3,4)&\{\emptyset,1245,1356,2346\}\\
6431:2b&0,0,e^{12},- e^{13},e^{23},e^{14}+e^{25}&7/20(1,1,2,3,3,4)&\{\emptyset,1245,1356,2346\}\\
6431:3&0,0,e^{12},e^{13},e^{23},e^{24}+e^{15}&7/20(1,1,2,3,3,4)&\{\emptyset,1245,135,234\}\\
6321:2&0,0,0,e^{12},e^{14},e^{15}+e^{23}&1/15(4,9,12,13,17,21)&\{\emptyset,1235,1346,2456\}\\
6321:4&0,0,0,- e^{12},e^{14}+e^{23},e^{15}+e^{34}&1/20(6,11,12,17,23,29)&\{\emptyset,125,146,2456\}\\
632:2&0,0,0,e^{12},e^{14},e^{24}+e^{13}&1/13(5,6,12,11,16,17)&\{\emptyset,1256,146,245\}\\
632:3a&0,0,0,e^{12},e^{23}+e^{14},e^{13}+e^{24}&3/7(1,1,2,2,3,3)&\{\emptyset,1256,146,245\}\\
632:3b&0,0,0,e^{12},e^{23}-e^{14},e^{24}+e^{13}&3/7(1,1,2,2,3,3)&\{\emptyset,1256,146,245\}\\
631:1&0,0,0,e^{12},e^{13},e^{14}&1/7(2,5,6,7,8,9)&\{\emptyset,1236,1256,134,145,23456,246,35\}\\
631:2&0,0,0,e^{12},e^{13},e^{24}&1/12(6,5,9,11,15,16)&\{\emptyset,1236,1256,1346,1456,2345,24,35\}\\
631:3&0,0,0,e^{12},e^{13},e^{23}+e^{14}&1/8(3,5,6,8,9,11)&\{\emptyset,1256,145,246\}\\
631:4&0,0,0,e^{12},e^{13},e^{24}+e^{15}&2/11(2,3,4,5,6,8)&\{\emptyset,1236,1346,24\}\\
631:5a&0,0,0,e^{12},e^{13},e^{35}+e^{24}&1/2(1,1,1,2,2,3)&\{\emptyset,1236,1256,1346,1456,2345,24,35\}\\
631:5b&0,0,0,- e^{12},e^{13},e^{24}+e^{35}&1/2(1,1,1,2,2,3)&\{\emptyset,1236,1256,1346,1456,2345,24,35\}\\
631:6&0,0,0,e^{12},e^{13},e^{34}+e^{25}&1/2(1,1,1,2,2,3)&\{\emptyset,1236,125,134,1456,2345,246,356\}\\
63:1&0,0,0,e^{12},e^{13},e^{23}&3/5(1,1,1,2,2,2)&\{\emptyset,123,1256,1346,145,2345,246,356\}\\
621:3&0,0,0,0,e^{12},e^{15}+e^{34}&1/12(5,8,9,9,13,18)&\{\emptyset,1236,1246,1345,15,2356,2456,34\}\\
62:3&0,0,0,0,e^{12},e^{13}+e^{24}&1/5(3,3,4,4,6,7)&\{\emptyset,1234,126,135,1456,2356,245,346\}\\
62:4a&0,0,0,0,e^{13}+e^{24},e^{34}+e^{12}&2/3(1,1,1,1,2,2)&\{\emptyset,1234,125,136,1456,2356,246,345\}\\
62:4b&0,0,0,0,e^{24}-e^{13},e^{34}+e^{12}&2/3(1,1,1,1,2,2)&\{\emptyset,1234,125,136,1456,2356,246,345\}\\
\end{longtable}
}
\end{footnotesize}

\begin{remark}
The only nonnice nilpotent Lie algebra of dimension $6$ also admits nilsoliton metrics: a Riemannian nilsoliton metric was constructed in \cite{Will:RankOne,LauretWill:OnTheDiagonalization}, and a nilsoliton metric of signature $(2,4)$ can be obtained by a Wick rotation of the form described in~\cite{Yan:Pseudo-RiemannianEinsteinhomogeneous}, namely by fixing the $\Z_2$-grading defined by the eigenspaces of the Nikolaevsky derivation and changing the sign of the $1$-eigenspace. However, in the absence of a canonical basis it does not make sense to restrict to diagonal metrics, so it is not possible to extend Theorem~\ref{thm:nice6} to include this case.
\end{remark}
\begin{remark}
The Ricci operator of the diagonal nilsoliton metrics constructed in Theorem~\ref{thm:nice6} is determined by~\eqref{eqn:normalizednil4} in terms of the diagonal Nikolayevsky derivation $N$; for this reason, we included $N$ in Table~\ref{table:Nilsoliton6} and the analogous tables for higher dimensions. For instance, for \texttt{31:1}, we obtain
\[\Ric=-\frac12\bigl(\id -\frac23(1,1,2)^D\bigr)=\bigl(-\frac16,-\frac16,\frac16\bigr)^D.\]
\end{remark}

The situation for dimension $7$ is different. To begin with, some nice Lie algebras do not admit a nilsoliton metric of type Nil4 because $N$ is zero. In addition, the polynomial equations \ref{enum:condP} need to be handled with more care in a few cases, in particular for the $5$ one-parameter families.
\begin{lemma}
\label{lemma:754321:9}
For the one-parameter family of Lie algebras
\[
\texttt{754321:9}\quad(0,0,(1-a) e^{12},e^{13},a e^{14}+e^{23},e^{15}+e^{24},e^{16}+e^{25}+e^{34}) \quad a\neq 0,1\]
the signatures of the diagonal metrics satisfying~\eqref{eqn:normalizednil4} are listed in Table~\ref{table:FamiliesDim7}.
\end{lemma}
\begin{proof}
Equation~\eqref{eq:condKnormalized} gives
\[X=(x_1,\dotsc, x_9)=(x_9,\frac{2}{5}-x_4,x_8,x_4,x_4,\frac35-x_4-x_8-x_9,\frac{2}{5}-x_8-x_9,x_8,x_9),\]
and condition~\ref{enum:condP} can be written in the form
\begin{equation}
\label{eqn:P1}
\frac{x_4^2}{(\frac25-x_4)x_6}=1, \qquad
\frac{x_8^2}{x_6x_7}=a^2, \qquad
\frac{x_9^2}{x_6x_7}=(1-a)^2,\end{equation}
with
\begin{equation}
\label{eqn:754321:9:linear}
x_4+x_6-x_7=\frac15, \qquad x_7+x_8+x_9=\frac25.
\end{equation}
The equations are invariant under the transformation
\[a\mapsto 1-a, \quad x_8\leftrightarrow x_9.\]
Accordingly, we will assume $a\geq \frac12$; the signatures occurring for $a<\frac12$ will be deduced exploiting the symmetry.

Notice that the value $x_4=\frac25$ is not allowed by \ref{enum:condH}. The first equation in~\eqref{eqn:P1} gives
\[x_6=\frac{5x_4^2}{2-5x_4};\]
dividing the second by the third gives
\[x_9=\epsilon \frac{1-a}{a}x_8, \quad \epsilon=\pm1.\]
The first linear relation in~\eqref{eqn:754321:9:linear} gives
\[x_7=-\frac15+x_4+x_6=\frac{15x_4-2}{5(2-5x_4)}.\]
We have
\[x_8+x_9=\frac{k}ax_8, \qquad k=a+\epsilon (1-a).\]
The possibility $k=0$ only occurs for $\epsilon=-1$, $a=\frac12$, giving
\[X=\left(\pm\frac3{5\sqrt{10}}, \frac4{25},\mp\frac3{5\sqrt{10}},\frac6{25},\frac6{25},\frac9{25},\frac25,\mp\frac3{5\sqrt{10}},\pm\frac3{5\sqrt{10}}\right).\]
Otherwise, we have $k\neq0$. Then the second linear relation in~\eqref{eqn:754321:9:linear} implies
\[x_8=\frac ak\left(\frac25-x_7\right)=
\frac{a(6-25x_4)}{5k(2 -5 x_4)}
.\]
In particular, the value $x_4=\frac6{25}$ is not allowed. The middle equation in~\eqref{eqn:P1} then gives
\begin{equation}
 \label{eqn:eqterzogradoP1}
 p(x_4)=
 25 (2-15x_4)x_4^2+\frac{(6-25 x_4)^2}{k^2}=0.
\end{equation}
For $\epsilon=1$, we have $k=1$ and~\eqref{eqn:eqterzogradoP1} is a constant-coefficient equation in $x_4$ with three solutions; the solution $x_4=\frac25$ must be discarded by~\ref{enum:condH}, so we are left with $x_4=\frac15,\frac65$, giving
\begin{gather*}
X=
\left(\frac{1-a}{5},\frac{1}{5},\frac{a}{5},\frac{1}{5},\frac{1}{5},\frac{1}{5},\frac{1}{5},\frac{a}{5},\frac{1-a}{5}\right),
\left(\frac{6 (1-a)}{5},-\frac{4}{5},\frac{6 a}{5},\frac{6}{5},\frac{6}{5},-\frac{9}{5},-\frac{4}{5},\frac{6 a}{5},\frac{6(1-a)}{5}\right).
\end{gather*}
For $\epsilon=-1$, we have
\begin{multline*}
X=\biggl(\frac{(a-1) (25 x_4-6)}{5 (2 a-1) (5 x_4-2)},\frac{2-5x_4}{5},\frac{a (25 x_4-6)}{5 (2 a-1) (5 x_4-2)},x_4,x_4,\\
\frac{5 x_4^2}{2-5 x_4},\frac{2-15 x_4}{5 (5 x_4-2)},\frac{a (25 x_4-6)}{5 (2 a-1) (5 x_4-2)},\frac{(a-1) (25 x_4-6)}{5 (2 a-1) (5 x_4-2)}\biggr).
\end{multline*}
In this case, $x_4=\frac25$, $x_4=\frac6{25}$ are not roots of $p$, any root of $p$ satisfies $x_4>\frac2{15}$, and $a\neq 0,1$ by assumption; therefore, any root of $p$ determines a vector $X$ sastisfying the conditions of Corollary~\ref{cor:KHLP}.
The discriminant of $p$ is positive for
\[0<k^2<\alpha=\frac{1}{16} (123 \sqrt{41}-767)\sim 1.29,\]
$0$ for $k^2=\alpha$ and negative otherwise.
The entries of $X$ depend continuosly on $k$, so it suffices to determine their sign on one value in each interval $0<k^2<1$, $1<k^2<\alpha$, $\alpha<k^2$, or equivalently (given our assumption $a>\frac12$), $\frac12<a<1$, $1<a<\frac12+\beta$, $\frac12+\beta<a$, where we have set
\[\beta = \frac{1}{8} \sqrt{123 \sqrt{41}-767}\sim 0.57.\]
Observe that the signs of $x_2,x_8$ and $x_9$ determine the signs of the other entries. Using the symmetry, we obtain the list of sign configurations and signatures in Table~\ref{table:signs754321:9}.
\end{proof}

\begin{table}[thp]
\centering
\caption{Nilsoliton signatures for \texttt{754321:9}\label{table:signs754321:9}}
\begin{tabular}{CC|CCC|C}
\toprule
\epsilon & a & \sign x_2 & \sign x_8 & \sign x_9 & \mathbf{S} \\
\midrule
\multirow{2}{*}{$1$} & \multirow{2}{*}{$ a<0$} & - & - & + & \{1345,47\} \\
&  & + & - & + & \{12457, 234\} \\\cmidrule(rl){3-6}
\multirow{2}{*}{$1$} & \multirow{2}{*}{$0< a<1$} & - & + & + & \{125,237\}\\
&  & + & + & + &  \{\emptyset,1357\} \\\cmidrule(rl){3-6}
\multirow{2}{*}{$1$} & \multirow{2}{*}{$a>1$} & - & + & - & \{123467, 2456\} \\
&  & + & + & - & \{146,34567\} \\\cmidrule(rl){3-6}
-1 & a<\frac{1}{2}-\beta \vee  \frac{1}{2}+\beta<a & + & + & + &  \{\emptyset,1357\} \\\cmidrule(rl){3-6}
\multirow{2}{*}{$-1$} & \multirow{2}{*}{$\frac{1}{2}-\beta< a<0 \vee 1< a\leq\frac{1}{2}+\beta$} & + & + & + &  \{\emptyset,1357\}\\
&  & - & + & + & \{125,237\} \\\cmidrule(rl){3-6}
\multirow{3}{*}{$-1$} & \multirow{3}{*}{$0< a<\frac{1}{2}$} & + & + & - & \{146,34567\} \\
&  & + & - & + & \{12457,234\} \\
&  & - & - & + & \{1345,47\} \\\cmidrule(rl){3-6}
\multirow{2}{*}{$-1$} & \multirow{2}{*}{$a=\frac{1}{2}$} & + & - & + & \{12457,234\} \\
&  & + & + & - & \{146,34567\} \\\cmidrule(rl){3-6}
\multirow{3}{*}{$-1$} & \multirow{3}{*}{$\frac{1}{2}< a<1$} & + & + & - & \{146,34567\} \\
&  & + & - & + & \{12457,234\} \\
&  & - & + & - & \{123467,2456\} \\
\bottomrule
\end{tabular}
\end{table}

\begin{lemma}\label{lemma:7421:14}
For the one-parameter families of Lie algebras
\begin{gather*}
\texttt{7421:14} \quad (0,0,0, (a-1) e^{12}, a e^{13},e^{14}+e^{23},e^{16}+e^{25}+e^{34}),\quad a\neq0,1 \\
\texttt{7431:13a} \quad (0,0,0,(A-1)e^{12},e^{14} +e^{23},Ae^{13}+e^{24} ,e^{15}+e^{26} +e^{34}), \quad A\neq 0,1 \\
\texttt{7431:13b} \quad (0,0,0,(A-1)e^{12},e^{14} -e^{23},Ae^{13}+e^{24} ,e^{15}+e^{26} +e^{34}), \quad A\neq 0,1
\end{gather*}
the signatures of the diagonal metrics satisfying~\eqref{eqn:normalizednil4} are listed in Table~\ref{table:FamiliesDim7}.
\end{lemma}
\begin{proof}
For \texttt{7421:14}, we have
\[X=\left(\frac{1}{19}+x_7,x_2,\frac{4}{19},\frac{7}{19}-x_7-x_2,\frac{8}{19}-x_7-x_2,\frac1{19}+x_2,x_7\right);\]
this gives the equations
\begin{equation}
\label{eqn:742114}
x_2\left(\frac1{19}+x_2\right)=x_4\left(x_4+\frac1{19}\right)a^2, \qquad x_7\left(\frac1{19}+x_7\right)=x_4\left(x_4+\frac1{19}\right)(a-1)^2,
\end{equation}
where
\begin{equation}
\label{eqn:7421:14:linear}
x_2+x_4+x_7=\frac7{19}.
\end{equation}
By~\eqref{eqn:742114}, if one of $x_2,x_4,x_7$ is in the interval $(-\frac1{19},0)$, then so are the others, but this contradicts~\eqref{eqn:7421:14:linear}. Therefore, $\logsign X$ is determined by the signs of $x_2,x_4,x_7$. Notice also that replacing $a$ with $\eta(a)=\frac1{1-a}$ has the effect of cycling through the variables $x_2,x_7,x_4$. Since $\eta$ is an order three transformation cycling through the intervals $(-\infty,0)$, $(0,1)$ and $(1,+\infty)$, it suffices to determine the sign configurations for one interval, and deduce the other cases by cylicity.

Assume $a>1$. Using~\eqref{eqn:7421:14:linear}, we eliminate $x_2$ from~\eqref{eqn:742114} obtaining
\[\begin{cases}
x_7 (19x_7+1)=(a-1)^2 x_4 (19x_4+1),\\
(19 x_4+19 x_7-8) (19 x_4+19 x_7-7)=19 a^2 x_4 (19x_4+1)
\end{cases}.\]
Solving for $x_4^2$ in the first equation and substituting in the second, we obtain
\[\begin{cases}
x_4^2=\frac{-a^2 x_4+2 a x_4-x_4+19 x_7^2+x_7}{19 (a-1)^2},\\
x_4 =\dfrac{a(152 x_7-28)+361 x_7^2-133 x_7+28}{19(a -1) (19 x_7-8)}\\
\end{cases};\]
notice that $x_7=8/19$ does not give a solution of~\eqref{eqn:742114} because of our assumption on $a$. Taking the first equation minus the second squared and substituting, we obtain
the following third-degree polynomial in $x_7$:
\begin{multline*}
p(x_7)=1008 - 2016 a + 1008 a^2+ x_7(-10260 + 19304 a - 10260 a^2)\\
+ x_7^2(25992 - 68951 a + 25992 a^2) +116603 a x_7^3=0.
\end{multline*}

For $a=2$, we see by an explicit computation that $p$ has three real roots; the corresponding sign configurations are
\begin{equation}
 \label{eqn:742114:possiblelogsignX}
x_2>0, x_4>0, x_7>0; \qquad x_2>0, x_4>0, x_7<0; \qquad x_2>0, x_4<0, x_7>0.
\end{equation}
More generally, for any $a>1$ the discriminant of $p$ is positive, so there are three distinct real roots of $p$, depending continuously on $a$; the corresponding values of $x_2$ and $x_4$ also depend continuously on $a$. Since zero and $-\frac{1}{19}$ are not roots of $p$, and $x_2,x_4$ are nonzero by~\eqref{eqn:742114}, we see that the three possibilities~\eqref{eqn:742114:possiblelogsignX} occur for all $a>1$.

Cycling with $\eta$, we obtain for $a<0$ three analogous configurations with $x_7$ positive, and three configurations with $x_4>0$ for $0<a<1$.

Similarly, for \texttt{7431:13a} and \texttt{7431:13b} we compute
\[
X=\left(\frac{1}{11}+x_8,\frac{2}{11},\frac{3}{11}-x_4-x_8,x_4,\frac{2}{11},\frac{4}{11}-x_4-x_8,\frac{1}{11}+x_4,x_8\right)\]
resulting in the system
\begin{equation*}
\begin{cases}
11x_8(1+11x_8)= 11x_3(1+11x_3)(A-1)^2\\
11x_4(1+11x_4)= 11x_3(1+11x_3)A^2\\
11x_3+11x_4+11x_8=3
\end{cases}
\end{equation*}
Again, replacing $A$ with $\eta(A)=\frac1{1-A}$ has the effect of cycling through the variables $x_3,x_4,x_8$.

Assume $A>1$. Proceeding as in the first case, we obtain
\[\begin{cases}
x_3= \frac{44 A x_8-6 A+121 x_8^2-33 x_8+6}{11 (A-1) (11 x_8-4)}\\
11979 A x_8^3+121 (20 - 47 A + 20 A^2)x_8^2-22 (35 - 62
A + 35 A^2)x_8+60 (-1 + A)^2=0\end{cases}\]
By the assumption $A>1$, $p$ has three real roots; the corresponding sign configurations are
\begin{equation*}
x_3>0, x_4>0, x_8>0; \qquad x_3>0, x_4>0, x_8<0; \qquad x_3<0, x_4>0, x_8>0.
\end{equation*}
Cycling with $\eta$, we obtain for $A<0$ three analogous configurations with $x_8$ positive, and three with $x_3>0$ for $0<A<1$. Notice that $\logsign X$ is only in the image of $M_{\Delta,2}$ when all entries are positive or when $x_8<0$.
\end{proof}

\begin{lemma}
\label{lemma:741:6}
For the one-parameter family of Lie algebras
\[
\texttt{741:6} \quad (0,0,0,(a-1)e^{12},ae^{13},e^{23},e^{16}+ e^{25}+ e^{34})\quad a\neq0,1 \]
the signatures of the diagonal metrics satisfying \eqref{eqn:normalizednil4} are listed in Table~\ref{table:FamiliesDim7}.
\end{lemma}
\begin{proof}
We compute
\begin{equation*}
X=\left(x_6,x_5,\frac{1}{2}-x_6-x_5,\frac{1}{2}-x_6-x_5,x_5,x_6\right).
\end{equation*}
As in Lemma~\ref{lemma:7421:14}, we can add a variable to obtain the more symmetric system of equations
\[\begin{cases}
{x_6^{2}}=(a-1)^{2}x_3^{2}\\{x_5^{2}}=a^{2}x_3^{2}\\
x_3+x_5+x_6=\frac{1}{2}
\end{cases}.\]
An easy computation shows that the solutions are
\[\begin{cases}
x_3=\frac{1}{4}\\
x_5=\frac{a}{4}\\
x_6=\frac{1-a}{4}
\end{cases} \vee
\begin{cases}
x_3=\frac{1}{4(1-a)}\\
x_5=\frac{a}{4 (a-1)}\\
x_6=\frac{1}{4}
\end{cases} \vee
\begin{cases}
x_3=\frac{1}{4 a}\\
x_5=\frac{1}{4}\\
x_6=\frac{a-1}{4 a}
\end{cases}.\]
For any values of $a\neq0,1$, it turns out that the possible signs of $x_3,x_5,x_6$ are those listed in Table~\ref{table:signs741:6}; the associated signatures are also collected in Table~\ref{table:FamiliesDim7}
\end{proof}
\begin{table}[thp]
\centering
\caption{Nilsoliton signatures for \texttt{741:6}\label{table:signs741:6}}
\begin{tabular}{C|CCC|C}
\toprule
a & \sign x_3 & \sign x_5 & \sign x_6 & \mathbf{S} \\
\midrule
\multirow{3}{*}{$ a<0$} & + & - & + & \{12357,126,13456,147,234,24567, 367, 5\} \\
 & + & + & + & \{\emptyset,1237,1256,1346,1457,2345,2467,3567\} \\
 & - & + & + & \{12367,125,134,14567,23456,247,357,6\} \\\cmidrule(rl){2-5}
\multirow{3}{*}{$ 0<a<1$} & + & + & + & \{\emptyset,1237,1256,1346,1457,2345,2467,3567\} \\
 & + & - & + & \{12357,126,13456,147,234,24567, 367, 5\} \\
 & + & + & - & \{12347,12456,136,157,235,267,34567,4\} \\\cmidrule(rl){2-5}
\multirow{3}{*}{$1<a$} & + & + & - & \{12347,12456,136,157,235,267,34567,4\} \\
 & - & + & + & \{12367,125,134,14567,23456,247,357,6\} \\
 & + & + & + & \{\emptyset,1237,1256,1346,1457,2345,2467,3567\} \\
\bottomrule
\end{tabular}
\end{table}

Beside the one-parameter families, there are three nice Lie algebras with corank $M_\Delta$ greater than one which must be studied separately:
\begin{lemma}
\label{lemma:75421:4}
For the Lie algebras
\begin{gather*}
\texttt{75421:4}\quad(0,0,e^{12},e^{13},e^{23},e^{15}+e^{24},e^{16}+e^{34})\\
\texttt{74321:12}\quad (0,0,0,-e^{12},e^{14}+e^{23},e^{15}+e^{34},e^{16}+e^{35})\\
\texttt{75432:3}\quad (0,0,e^{12},e^{13},e^{14}+e^{23},e^{15}+e^{24},e^{25}-e^{34})
\end{gather*}
the signatures of the diagonal metrics satisfying~\eqref{eqn:normalizednil4} are listed in Table~\ref{table:Nilsoliton7}.
\end{lemma}
\begin{proof}
For \texttt{75421:4}, we compute
\begin{equation}
\label{eqn:X754214}
X=\left(x_7,\frac{2}{5}-x_3,x_3,x_3,\frac35-x_3-x_7,\frac{2}{5}-x_7,x_7\right),
\end{equation}
which determines the system of equations
\begin{equation}
\label{eqn:X754214:system}
\begin{cases}
x_3^2=(\frac25-x_3)x_5\\x_7^2=(\frac25-x_7)x_5\\
x_3+x_5+x_7=\frac35\end{cases}.
\end{equation}
This implies
\[\frac{x_3^2}{\frac25-x_3}=\frac{x_7^2}{\frac25-x_7},\]
giving either $x_7=x_3$ or $x_7=\frac{2 x_3}{-2 + 5 x_3}$. The latter however implies
\[x_3+x_5+x_7=x_3+\frac{x_3^2}{\frac25-x_3}-\frac{\frac25x_3}{\frac25-x_3}=0,\]
which contradicts~\eqref{eqn:X754214:system}. We can therefore assume $x_7=x_3$, reducing~\eqref{eqn:X754214:system} to
\[\frac{x_3^2}{\frac25-x_3}+2x_3=\frac35,\]
which has solutions $x_3=\frac15$, $x_3=\frac65$. The case $x_3=\frac15$ gives the signatures $\{\emptyset,12457,1357,234\}$; the case $x_3=\frac65$ gives $\{125,1345, 237,47\}$.

For $\texttt{74321:12}$,
we have
\[X=\left(\frac25-x_3,x_7,x_3,x_3,\frac35-x_3-x_7,\frac{2}{5}-x_7,x_7\right)\]
which is~\eqref{eqn:X754214} with $x_1$ and $x_2$ interchanged. Thus we get the solutions $x_3=x_7=\frac15$ (signatures $\{\emptyset,1257,146,24567\}$) and $x_3=x_7=\frac65$ (signatures $\{123467,135,237,3456\}$).

For \texttt{75432:3}, we compute
\[X=\left( x_8,\frac{1}{5}+x_6,\frac{2}{5}-x_8,\frac{1}{5}-x_6,\frac{1}{5}-x_6,x_6,\frac{2}{5}-x_8,x_8\right),\]
yielding two independent equations
\[\left(\frac15+x_6\right)x_6=\pm \left(\frac15-x_6\right)^2, \qquad x_8^2=\left(\frac25-x_8\right)^2,
\]
which are only satisfied for $x_6=\frac1{15}$, $x_8=\frac15$; this gives $X$ with all-positive entries and signatures $\{\emptyset,1357\}$.
\end{proof}

\begin{remark}
\label{remark:accorgimenti}
The nice Lie algebras in our tables are based on the classification of~\cite{ContiRossi:Construction} with the following minor differences.

For three families of Lie algebras, namely \texttt{7431:13b}, \texttt{97654321:43} and \texttt{9641:92b}, we had to correct the sign of an entry.

In addition, due to a limitation in the software used to produce the tables, we have not always used the same exact set of parameters as in~\cite{ContiRossi:Construction} for Lie algebras of dimension $8$ and $9$. However, the parameters taken here only differ by an affine transformation.
\end{remark}

\begin{theorem}
\label{thm:nice7}
The $7$-dimensional irreducible nice nilpotent Lie algebras that do not admit a diagonal nilsoliton metric of type Nil4 are exactly those contained in Table~\ref{table:ciccia}. For each of the others, the signatures of diagonal metrics satisfying~\eqref{eqn:normalizednil4} are listed in Table~\ref{table:FamiliesDim7} and Table~\ref{table:Nilsoliton7}.
\end{theorem}
\begin{proof}
Similar as Theorem~\ref{thm:nice6}, with two differences. In some cases, listed in Table~\ref{table:ciccia}, there is no vector $X$ satisfying the conditions of Corollary~\ref{cor:KHLP}; the last column of the table indicates which of the four conditions fails in each case.

In addition, there are five one-parameter familes of nice nilpotent Lie algebras in dimension $7$. For each of them, \ref{enum:condP} entails two or three polynomial equations depending on a parameter; a case-by-case analysis is required to compute the possible signatures (see Lemmas~\ref{lemma:754321:9}, \ref{lemma:7421:14} and~\ref{lemma:741:6}).
Among the isolated nice nilpotent Lie algebras (i.e. outside the above-mentioned families), there are three nice nilpotent Lie algebras of dimension $7$ with $\dim\coker M_\Delta=2$; \ref{enum:condP} gives then a system of two polynomial equations of degree two which is solved explicitly in Lemma~\ref{lemma:75421:4}.

The other cases can be handled automatically using~\cite{demonblast} as in Theorem~\ref{thm:nice6}.
\end{proof}

\begin{remark}
\label{remark:notnecessarilydiagonal}
It is known (see~\cite{Nikolayevsky}) that a nice nilpotent Lie algebra admits a Riemannian nilsoliton metric if and only if it admits a diagonal Riemannian nilsoliton metric.

As a consequence of the classification, we see that this does not hold for arbitrary signature $(p,q)$. Indeed, the nice Lie algebras \texttt{731:21} and \texttt{731:24b} are isomorphic as Lie algebras, but the signatures of the diagonal nilsoliton metrics that they carry are different. In this example, there is no Riemannian nilsoliton metric.

Another example is the Lie algebra with two nice bases corresponding to \texttt{731:19} and \texttt{731:22a}; in this case there is a Riemannian nilsoliton metric, but only \texttt{731:19} admits indefinite diagonal nilsoliton metrics.
\end{remark}

\begin{remark}
It is evident from Table~\ref{table:Nilsoliton7} that a
nice nilpotent Lie algebra may admit an indefinite nilsoliton metric even if it is not a Riemannian nilsoliton. In particular, we see that \texttt{74321:7} and \texttt{7421:9} admit a Lorentzian nilsoliton metric but not a Riemannian nilsoliton metric.
\end{remark}

\begin{footnotesize}
{\setlength{\tabcolsep}{2pt}
\begin{longtable}[c]{>{\ttfamily}c  C C }
\caption{Signatures of diagonal nilsoliton metrics of type Nil4 on families of nice nilpotent Lie algebras of dimension $7$
\label{table:FamiliesDim7}}\\
\toprule
\textnormal{Name} & \g  & \mathbf{S} \\
\midrule
\endfirsthead
\multicolumn{3}{c}{\tablename\ \thetable\ -- \textit{Continued from previous page}} \\
\toprule
\textnormal{Name} & \g  & \mathbf{S} \\
\midrule
\endhead
\bottomrule\\[-7pt]
\multicolumn{3}{c}{\tablename\ \thetable\ -- \textit{Continued to next page}} \\
\endfoot
\bottomrule\\[-7pt]
\endlastfoot
754321:9 & \begin{array}{c}
0,0,(1-a) e^{12},\\
e^{13},a e^{14}+e^{23},\\
e^{15}+e^{24},e^{16}+e^{25}+e^{34}
\end{array} &
\begin{array}{rl}
a<\frac12-\beta & \{\emptyset,12457,1345,1357,234,47\}\\
\frac12-\beta \leq a<0& \{\emptyset,12457,125,1345,1357,234,237,47\}\\
0<a<\frac12 & \{\emptyset,12457,125,146,1345,\\
&1357,234,237,34567,47\}\\
a=\frac12 & \{\emptyset,12457,125,1357,146,234,237,34567\}\\
\frac12<a<1 & \{\emptyset,123467,12457,125,1357,\\
&146,234,237,2456,34567\}\\
1<a\leq\frac12+\beta & \{\emptyset,123467,125,1357,146,237,2456,34567\}\\
\frac12+\beta<a & \{\emptyset,123467,1357,146,2456,34567\}\\
\end{array}\\
&&\alpha=\frac{1}{16} (123 \sqrt{41}-767),\quad \beta = \frac{1}{8} \sqrt{123 \sqrt{41}-767}\\\cmidrule(lr){3-3}
\texttt{7431:13a} & \begin{array}{c}
0,0,0,(A-1)e^{12},e^{14} +e^{23},\\
Ae^{13}+e^{24} ,e^{15}+e^{26} +e^{34}\end{array} & \begin{array}{rl}
A<0 & \{\emptyset,1256,1467,2457\}\\
A>0,A\neq 1 & \{\emptyset,12347,1256,135,1467,2457,236,34567\}\\
\end{array}\\\cmidrule(lr){3-3}
7431:13b & \begin{array}{c}
0,0,0,(A-1)e^{12},e^{14} -e^{23},\\
Ae^{13}+e^{24} ,e^{15}+e^{26} +e^{34}
\end{array}  & \begin{array}{rl}
A<0 & \{\emptyset,1256,1467,2457\}\\
A>0,A\neq 1 & \{\emptyset,12347,1256,135,1467,236,2457,34567\}\\
\end{array}\\\cmidrule(lr){3-3}
7421:14 & \begin{array}{c}
0,0,0, (a-1) e^{12}, a e^{13},\\
e^{14}+e^{23},e^{16}+e^{25}+e^{34}
\end{array}
& \begin{array}{rl}
a<0 & \{
\emptyset,12367,1256,126,134,1457,\\
&147,23456,24567,2467,357,5\}\\
0<a<1 &
\{\emptyset,12367,1256,126,134,1457,\\
&147,23456,24567,2467,357,5\}\\
1<a & \{
\emptyset,12347,12367,1256,134,136, \\
&1457,23456,235,2467,34567,357\}
\end{array}\\\cmidrule(lr){3-3}
741:6 & \begin{array}{c}
0,0,0,(a-1)e^{12},ae^{13},\\
e^{23},e^{16}+ e^{25}+ e^{34}
\end{array}  & \begin{array}{rl}
 a<0 & \{\emptyset,12357,12367,1237, 126,125,1256,134,\\
& 13456,1346,14567,1457,147,234,2345,23456,\\
& 24567,2467,247,3567,357,367,5,6\} \\
0<a<1 & \{ \emptyset,12347,1235,12357,12456,1256,126,13456,\\
& 1346,136,1457,147,157,234,2345,235,\\
& 24567,2467,267,34567,3567,367,4,5\}\\
1<a & \{\emptyset,12347,12367,1237,12456,125,1256,134,\\
& 1346,136,14567,1457,157,2345,23456,235,\\
& 2467,247,267,34567,3567,357,4,6\} \\
\end{array}\\
\end{longtable}
}
\end{footnotesize}

{\setlength{\tabcolsep}{2pt}
\begin{longtable}[c]{>{\ttfamily}c C C c}
\caption{Nice nilpotent Lie algebras of dimension $7$ that do not admit a diagonal nilsoliton metric of type Nil4\label{table:ciccia}}\\
\toprule
\textnormal{Name} & \g &   N & Obstruction \\
\midrule
\endfirsthead
\multicolumn{4}{c}{\tablename\ \thetable\ -- \textit{Continued from previous page}} \\
\toprule
\textnormal{Name} & \g &   N & Obstruction \\
\midrule
\endhead
\bottomrule\\[-7pt]
\multicolumn{4}{c}{\tablename\ \thetable\ -- \textit{Continued to next page}} \\
\endfoot
\bottomrule\\[-7pt]
\endlastfoot
754321:5&0,0,- e^{12},e^{13},e^{14},e^{15},e^{34}+e^{16}+e^{25}&1/5(1,2,3,4,5,6,7)&\ref{enum:condH}\\
754321:6&0,0,e^{12},e^{13},e^{14}+e^{23},e^{15}+e^{24},e^{16}+e^{34}&1/5(1,2,3,4,5,6,7)&\ref{enum:condH}\\
754321:7&0,0,e^{12},e^{13},e^{14}+e^{23},e^{15}+e^{24},e^{25}+e^{16}&1/5(1,2,3,4,5,6,7)&\ref{enum:condH}\\
75421:6&0,0,- e^{12},e^{13},e^{23},e^{24}+e^{15},e^{35}+e^{26}+e^{14}&1/5(2,1,3,5,4,6,7)&\ref{enum:condH}\\
74321:11&0,0,0,- e^{12},e^{14},e^{24}+e^{15},e^{45}+e^{26}+e^{13}&1/5(1,2,6,3,4,5,7)&\ref{enum:condH}\\
74321:15&0,0,0,- e^{12},e^{14}+e^{23},e^{15}+e^{34},e^{16}+e^{35}+e^{24}&1/5(1,3,2,4,5,6,7)&\ref{enum:condH}\\
7431:2&0,0,0,e^{12},e^{14},e^{13}+e^{24},e^{15}&1/4(1,2,4,3,4,5,5)&\ref{enum:condH}\\
7431:6a&0,0,0,- e^{12},e^{14}+e^{23},e^{13}+e^{24},e^{34}+e^{15}&4/11(1,1,2,2,3,3,4)&\ref{enum:condP}\\
7431:6b&0,0,0,- e^{12},e^{23}-e^{14},e^{24}+e^{13},e^{34}+e^{15}&4/11(1,1,2,2,3,3,4)&\ref{enum:condP}\\
741:3a&0,0,0,e^{12},e^{13},e^{23},e^{24}+e^{35}&1/2(1,1,1,2,2,2,3)&\ref{enum:condH}\\
741:3b&0,0,0,- e^{12},e^{13},e^{23},e^{24}+e^{35}&1/2(1,1,1,2,2,2,3)&\ref{enum:condH}\\
741:4&0,0,0,e^{12},e^{13},e^{23},e^{34}+e^{25}&1/2(1,1,1,2,2,2,3)&\ref{enum:condH}\\
731:8&0,0,0,0,e^{12},e^{34},e^{13}+e^{25}&1/3(2,1,2,2,3,4,4)&\ref{enum:condH}\\
731:16a&0,0,0,0,e^{12},e^{13},e^{25}+e^{36}+e^{14}&1/2(1,1,1,2,2,2,3)&\ref{enum:condH}\\
731:16b&0,0,0,0,- e^{12},e^{13},e^{36}+e^{14}+e^{25}&1/2(1,1,1,2,2,2,3)&\ref{enum:condH}\\
731:18&0,0,0,0,e^{12},e^{13},e^{35}+e^{26}+e^{14}&1/2(1,1,1,2,2,2,3)&\ref{enum:condH}
\end{longtable}
}

\FloatBarrier

\begin{landscape}
{\setlength{\tabcolsep}{2pt}
\begin{longtable}[c]{>{\ttfamily}c C C C}
\caption{Irreducible nice nilpotent Lie algebras of dimension $7$ that admit a diagonal nilsoliton metric of type Nil4\label{table:Nilsoliton7}}\\
\toprule
\textnormal{Name} & \g &   N &  \mathbf{S} \\
\midrule
\endfirsthead
\multicolumn{4}{c}{\tablename\ \thetable\ -- \textit{Continued from previous page}} \\
\toprule
\textnormal{Name} & \g &   N &  \mathbf{S} \\
\midrule
\endhead
\bottomrule\\[-7pt]
\multicolumn{4}{c}{\tablename\ \thetable\ -- \textit{Continued to next page}} \\
\endfoot
\bottomrule\\[-7pt]
\endlastfoot
754321:1&0,0,e^{12},e^{13},e^{14},e^{15},e^{16}&2/37(1,16,17,18,19,20,21)&\{\emptyset,1246,1357,234567\}\\
754321:2&0,0,e^{12},e^{13},e^{14},e^{15},e^{23}+e^{16}&5/34(1,4,5,6,7,8,9)&\{\emptyset,1357\}\\
754321:3&0,0,e^{12},e^{13},e^{14},e^{23}+e^{15},e^{24}+e^{16}&17/100(1,3,4,5,6,7,8)&\{\emptyset,1246\}\\
754321:9&0,0,- e^{12} {(-1+\lambda)},e^{13},e^{23}+ e^{14} \lambda,e^{15}+e^{24},e^{25}+e^{16}+e^{34}&1/5(1,2,3,4,5,6,7)&\text{see Table~\ref{table:FamiliesDim7}}\\
75432:1&0,0,- e^{12},e^{13},e^{14},e^{15},e^{34}+e^{25}&7/52(1,4,5,6,7,8,11)&\{\emptyset,12467,1357,23456\}\\
75432:2&0,0,- e^{12},e^{13},e^{14},e^{15}+e^{23},e^{34}+e^{25}&5/31(1,3,4,5,6,7,9)&\{\emptyset,12467\}\\
75432:3&0,0,- e^{12},e^{13},e^{14}+e^{23},e^{24}+e^{15},e^{34}+e^{25}&1/5(1,2,3,4,5,6,7)&\{\emptyset,1357\}\\
75421:1&0,0,e^{12},e^{13},e^{23},e^{14},e^{16}&1/18(3,10,13,16,23,19,22)&\{\emptyset,12457,1356,23467\}\\
75421:2&0,0,e^{12},e^{13},e^{23},e^{14},e^{16}+e^{25}&45/353(2,3,5,7,8,9,11)&\{\emptyset,23467\}\\
75421:3&0,0,- e^{12},e^{13},e^{23},e^{14},e^{26}+e^{34}&1/52(10,23,33,43,56,53,76)&\{\emptyset,12457,13567,2346\}\\
75421:4&0,0,e^{12},e^{13},e^{23},e^{15}+e^{24},e^{16}+e^{34}&1/5(1,2,3,4,5,6,7)&\{\emptyset,12457,125,1345,1357,234, 237,47\}\\
75421:5a&0,0,- e^{12},e^{13},e^{23},e^{25}+e^{14},e^{26}+e^{34}&19/65(1,1,2,3,3,4,5)&\{\emptyset,12457,13567,2346\}\\
75421:5b&0,0,- e^{12},- e^{13},e^{23},e^{14}+e^{25},e^{26}+e^{34}&19/65(1,1,2,3,3,4,5)&\{\emptyset,12457,13567,2346\}\\
7542:1&0,0,e^{12},e^{13},e^{23},e^{14},e^{25}&9/28(1,1,2,3,3,4,4)&\{\emptyset,1245,13567,23467\}\\
7542:2&0,0,e^{12},e^{13},e^{23},e^{14},e^{15}+e^{24}&1/12(3,5,8,11,13,14,16)&\{\emptyset,1245,1356,2346\}\\
7542:3a&0,0,e^{12},e^{13},e^{23},e^{24}+e^{15},e^{14}+e^{25}&9/28(1,1,2,3,3,4,4)&\{\emptyset,1245,1357,2347\}\\
7542:3b&0,0,e^{12},- e^{13},e^{23},-e^{15}+e^{24},e^{14}+e^{25}&9/28(1,1,2,3,3,4,4)&\{\emptyset,1245,1357,2347\}\\
74321:2&0,0,0,e^{12},e^{14},e^{15},e^{16}+e^{23}&1/27(5,17,20,22,27,32,37)&\{\emptyset,1257,146,24567\}\\
74321:5&0,0,0,- e^{12},e^{14},e^{15}+e^{23},e^{16}+e^{34}&1/5(1,3,3,4,5,6,7)&\{\emptyset,12357,1346,24567\}\\
74321:6&0,0,0,- e^{12},e^{14},e^{15}+e^{23},e^{45}+e^{26}&1/60(16,21,48,37,53,69,90)&\{\emptyset,12357,13467,2456\}\\
74321:7&0,0,0,- e^{12},e^{14},e^{15},e^{26}+e^{13}+e^{45}&1/5(1,2,6,3,4,5,7)&\{12357,13467,23456,3\}\\
74321:10&0,0,0,- e^{12},e^{14},e^{23}+e^{15},e^{13}+e^{45}+e^{26}&8/47(2,1,6,3,5,7,8)&\{12357,13467\}\\
74321:12&0,0,0,- e^{12},e^{23}+e^{14},e^{15}+e^{34},e^{16}+e^{35}&1/5(1,3,2,4,5,6,7)&\{\emptyset,123467,1257,135,146,237,24567,3456\}\\
7431:3&0,0,0,e^{12},e^{14},e^{24},e^{23}+e^{15}&1/68(20,31,60,51,71,82,91)&\{\emptyset,12356,13467,2457\}\\
7431:4&0,0,0,e^{12},e^{14},e^{24}+e^{13},e^{15}+e^{23}&10/67(2,3,6,5,7,8,9)&\{12356,13467\}\\
7431:5&0,0,0,e^{12},e^{13}+e^{24},e^{14},e^{34}+e^{25}&1/12(5,4,8,9,13,14,17)&\{\emptyset,1256,1457,2467\}\\
7431:9&0,0,0,e^{12},e^{14},e^{24}+e^{13},e^{16}+e^{25}&1/7(2,3,6,5,7,8,10)&\{\emptyset,1256,146,245\}\\
7431:10a&0,0,0,e^{12},e^{14}+e^{23},e^{24}+e^{13},e^{15}+e^{26}&4/11(1,1,2,2,3,3,4)&\{\emptyset,1256,1467,2457\}\\
7431:10b&0,0,0,e^{12},-e^{14}+e^{23},e^{13}+e^{24},e^{26}+e^{15}&4/11(1,1,2,2,3,3,4)&\{\emptyset,1256,1467,2457\}\\
7431:11a&0,0,0,e^{12},e^{13}+e^{24},e^{14}+e^{23},e^{15}+e^{26}&4/11(1,1,2,2,3,3,4)&\{\emptyset,1256,145,246\}\\
7431:11b&0,0,0,e^{12},e^{24}-e^{13},e^{23}+e^{14},e^{15}+e^{26}&4/11(1,1,2,2,3,3,4)&\{\emptyset,1256,145,246\}\\
7431:12a&0,0,0,e^{12},e^{14},e^{13}+e^{24},e^{26}+e^{34}+e^{15}&4/11(1,1,2,2,3,3,4)&\{\emptyset,1256,1467,2457\}\\
7431:12b&0,0,0,e^{12},- e^{14},e^{13}+e^{24},e^{26}+e^{34}+e^{15}&4/11(1,1,2,2,3,3,4)&\{\emptyset,1256,1467,2457\}\\
7431:13a&0,0,0, {(-1+A)} e^{12},e^{23}+e^{14}, A e^{13}+e^{24},e^{34}+e^{26}+e^{15}&4/11(1,1,2,2,3,3,4)&\text{see Table~\ref{table:FamiliesDim7}}\\
7431:13b&0,0,0, {(-1+A)} e^{12},e^{23}-e^{14}, A e^{13}+e^{24},e^{34}+e^{15}+e^{26}&4/11(1,1,2,2,3,3,4)&\text{see Table~\ref{table:FamiliesDim7}}\\
7421:1&0,0,0,e^{12},e^{13},e^{14},e^{16}&2/13(1,5,6,6,7,7,8)&\{\emptyset,1236,1256,1347,1457,234567,2467,35\}\\
7421:2&0,0,0,e^{12},e^{13},e^{24},e^{26}&1/21(12,5,15,17,27,22,27)&\{\emptyset,1236,1256,13467,14567,23457,247,35\}\\
7421:3&0,0,0,e^{12},e^{13},e^{14},e^{16}+e^{23}&1/31(8,19,24,27,32,35,43)&\{\emptyset,1236,1347,2467\}\\
7421:4&0,0,0,e^{12},e^{13},e^{14},e^{16}+e^{24}&2/37(5,10,16,15,21,20,25)&\{\emptyset,1347,1457,35\}\\
7421:5&0,0,0,e^{12},e^{13},e^{24},e^{26}+e^{14}&1/19(10,5,14,15,24,20,25)&\{\emptyset,23457,247,35\}\\
7421:6&0,0,0,e^{12},e^{13},e^{14},e^{35}+e^{16}&2/27(3,10,8,13,11,16,19)&\{\emptyset,1347,1457,35\}\\
7421:7&0,0,0,e^{12},e^{13},e^{24},e^{35}+e^{26}&1/33(18,10,15,28,33,38,48)&\{\emptyset,13467,14567,35\}\\
7421:8&0,0,0,e^{12},e^{13},e^{24},e^{15}+e^{26}&1/39(15,14,27,29,42,43,57)&\{\emptyset,1256,13467,23457\}\\
7421:9&0,0,0,- e^{12},e^{13},e^{14}+e^{23},e^{16}+e^{34}&2/19(3,5,6,8,9,11,14)&\{126,147,24567,5\}\\
7421:10&0,0,0,e^{12},e^{13},e^{14}+e^{23},e^{25}+e^{16}&2/19(3,5,6,8,9,11,14)&\{\emptyset,1256,1457,2467\}\\
7421:11a&0,0,0,e^{12},e^{13},e^{14},e^{16}+e^{24}+e^{35}&5/17(1,2,2,3,3,4,5)&\{\emptyset,1347,1457,35\}\\
7421:11b&0,0,0,- e^{12},e^{13},e^{14},e^{35}+e^{16}+e^{24}&5/17(1,2,2,3,3,4,5)&\{\emptyset,1347,1457,35\}\\
7421:12&0,0,0,e^{12},e^{13},e^{14},e^{34}+e^{25}+e^{16}&2/19(3,5,6,8,9,11,14)&\{\emptyset,1256,1457,2467\}\\
7421:13&0,0,0,e^{12},e^{13},e^{24},e^{14}+e^{35}+e^{26}&20/139(4,2,3,6,7,8,10)&\{13467,14567\}\\
7421:14&0,0,0, e^{12} {(-1+\lambda)}, e^{13} \lambda,e^{23}+e^{14},e^{25}+e^{34}+e^{16}&2/19(3,5,6,8,9,11,14)&\text{see Table~\ref{table:FamiliesDim7}}\\
742:1&0,0,0,e^{12},e^{13},e^{14},e^{24}&1/11(4,5,9,9,13,13,14)&\{\emptyset,12367,12567,1347,1457,23456,246,35\}\\
742:2&0,0,0,e^{12},e^{13},e^{24},e^{14}+e^{23}&1/23(9,10,18,19,27,29,28)&\{\emptyset,12567,1456,247\}\\
742:3&0,0,0,e^{12},e^{13},e^{14},e^{15}&1/5(1,4,4,5,5,6,6)&\{\emptyset,12367,1256,1347,145,234567,246,357\}\\
742:4&0,0,0,e^{12},e^{13},e^{24},e^{15}&1/21(6,11,15,17,21,28,27)&\{\emptyset,12367,1256,13467,1456,23457,24,357\}\\
742:5&0,0,0,e^{12},e^{13},e^{24},e^{35}&1/12(6,5,5,11,11,16,16)&\{\emptyset,12367,12567,13467,14567,2345,24,35\}\\
742:6&0,0,0,e^{12},e^{13},e^{14},e^{15}+e^{23}&1/16(5,10,11,15,16,20,21)&\{\emptyset,1347,145,357\}\\
742:7&0,0,0,e^{12},e^{13},e^{24},e^{23}+e^{15}&1/41(11,22,30,33,41,55,52)&\{12367,1256,23457,24\}\\
742:8&0,0,0,e^{12},e^{13},e^{14},e^{24}+e^{15}&1/37(11,20,29,31,40,42,51)&\{\emptyset,12367,1347,246\}\\
742:9a&0,0,0,e^{12},e^{13},e^{14},e^{24}+e^{35}&1/9(3,5,5,8,8,11,13)&\{\emptyset,12367,12567,1347,1457,23456,246,35\}\\
742:9b&0,0,0,- e^{12},e^{13},e^{14},e^{35}+e^{24}&1/9(3,5,5,8,8,11,13)&\{\emptyset,12367,12567,1347,1457,23456,246,35\}\\
742:10&0,0,0,e^{12},e^{13},e^{14},e^{25}+e^{34}&1/20(7,10,12,17,19,24,29)&\{\emptyset,12367,1256,134,1457,23456,2467,357\}\\
742:11&0,0,0,e^{12},e^{13},e^{24},e^{14}+e^{35}&1/47(22,20,21,42,43,62,64)&\{\emptyset,12367,12567,35\}\\
742:12&0,0,0,e^{12},e^{13},e^{24},e^{34}+e^{25}&1/20(10,7,11,17,21,24,28)&\{\emptyset,12367,1256,1346,14567,2345,247,357\}\\
742:13&0,0,0,e^{12},e^{13},e^{23}+e^{14},e^{24}+e^{15}&5/29(2,3,4,5,6,7,8)&\{\emptyset,246\}\\
742:14a&0,0,0,e^{12},e^{13},e^{14}+e^{23},e^{24}+e^{35}&5/17(1,2,2,3,3,4,5)&\{12367,1347,23456,35\}\\
742:14b&0,0,0,- e^{12},e^{13},e^{14}+e^{23},e^{24}+e^{35}&5/17(1,2,2,3,3,4,5)&\{12367,1347,23456,35\}\\
742:15&0,0,0,e^{12},e^{13},e^{23}+e^{14},e^{34}+e^{25}&2/19(3,5,6,8,9,11,14)&\{12367,134,23456,357\}\\
742:16&0,0,0,e^{12},e^{13},e^{15}+e^{24},e^{14}+e^{35}&13/29(1,1,1,2,2,3,3)&\{\emptyset,12367\}\\
742:17&0,0,0,e^{12},e^{13},e^{15}+e^{24},e^{25}+e^{34}&1/13(5,6,7,11,12,17,18)&\{\emptyset,12367,1346,247\}\\
742:18a&0,0,0,e^{12},e^{13},e^{25}+e^{34},e^{24}+e^{35}&1/12(6,5,5,11,11,16,16)&\{\emptyset,12367,1257,1347,14567,2345,246,356\}\\
742:18b&0,0,0,- e^{12},e^{13},-e^{25}+e^{34},e^{24}+e^{35}&1/12(6,5,5,11,11,16,16)&\{\emptyset,12367,1257,1347,14567,2345,246,356\}\\
741:1&0,0,0,e^{12},e^{13},e^{23},e^{14}&1/13(5,7,9,12,14,16,17)&\{\emptyset,1237,12567,1346,145,23457,2467,356\}\\
741:2&0,0,0,e^{12},e^{13},e^{23},e^{24}+e^{15}&1/37(15,19,23,34,38,42,53)&\{\emptyset,1237,13467,246\}\\
741:5&0,0,0,e^{12},e^{13},e^{23},e^{15}+e^{36}+e^{24}&1/2(1,1,1,2,2,2,3)&\{\emptyset,1237\}\\
741:6&0,0,0, e^{12} {(-1+\lambda)}, e^{13} \lambda,e^{23},e^{25}+e^{16}+e^{34}&1/2(1,1,1,2,2,2,3)&\text{see Table~\ref{table:FamiliesDim7}}\\
7321:2&0,0,0,0,e^{12},e^{15},e^{34}+e^{16}&5/21(1,3,3,3,4,5,6)&\{\emptyset,12346,126,1357,1457,23567,24567,34\}\\
7321:5&0,0,0,0,e^{12},e^{15},e^{34}+e^{25}+e^{16}&4/27(2,4,5,5,6,8,10)&\{\emptyset,1357,1457,34\}\\
7321:7&0,0,0,0,e^{12},e^{13}+e^{25},e^{14}+e^{35}+e^{26}&2/19(5,3,6,9,8,11,14)&\{\emptyset,1246,1567,2457\}\\
732:3&0,0,0,0,e^{12},e^{15},e^{25}+e^{34}&1/21(8,11,15,15,19,27,30)&\{\emptyset,12367,12467,1357,1457,23456,256,34\}\\
732:5&0,0,0,0,e^{12},e^{13}+e^{25},e^{15}+e^{34}&1/41(22,15,30,29,37,52,59)&\{12367,13456,2357,34\}\\
732:6&0,0,0,0,e^{12},e^{15}+e^{23},e^{14}+e^{25}&7/16(1,1,2,2,2,3,3)&\{\emptyset,1267,157,256\}\\
731:5&0,0,0,0,e^{12},e^{13},e^{15}+e^{34}&1/13(5,9,9,10,14,14,19)&\{\emptyset,1237,12467,1345,156,23567,2457,346\}\\
731:6&0,0,0,0,e^{12},e^{13},e^{25}+e^{34}&1/11(6,5,7,9,11,13,16)&\{\emptyset,1237,12467,1357,14567,23456,25,346\}\\
731:7&0,0,0,0,e^{12},e^{24}+e^{13},e^{15}&1/13(5,7,12,10,12,17,17)&\{\emptyset,12347,1267,135,1456,23567,2457,346\}\\
731:9&0,0,0,0,e^{12},e^{13},e^{14}+e^{25}&2/13(3,3,5,6,6,8,9)&\{\emptyset,1237,1267,1357,1567,2356,25,36\}\\
731:10&0,0,0,0,e^{12},e^{13},e^{24}+e^{15}&1/5(2,3,4,4,5,6,7)&\{\emptyset,1237,1267,135,156,23567,257,36\}\\
731:11&0,0,0,0,e^{12},e^{24}+e^{13},e^{34}+e^{15}&1/23(10,13,18,15,23,28,33)&\{\emptyset,23567,2457,346\}\\
731:12&0,0,0,0,e^{12},e^{14}+e^{23},e^{25}+e^{13}&1/14(7,6,12,11,13,18,19)&\{\emptyset,1267,1457,2456\}\\
731:17&0,0,0,0,e^{12},e^{13},e^{36}+e^{24}+e^{15}&4/37(4,6,5,8,10,9,14)&\{\emptyset,1237,1267,36\}\\
731:19&0,0,0,0,e^{12},e^{34},e^{36}+e^{15}&1/12(5,8,5,8,13,13,18)&\{\emptyset,12347,12467,1356,15,23457,24567,36\}\\
731:20&0,0,0,0,- e^{12},e^{23}+e^{14},e^{35}+e^{16}&1/20(7,12,11,16,19,23,30)&\{\emptyset,12347,126,1356,1457,235,24567,3467\}\\
731:21&0,0,0,0,e^{12},e^{34},e^{46}+e^{13}+e^{25}&4/11(2,1,2,1,3,3,4)&\{12347,12367,13457,13567\}\\
731:22a&0,0,0,0,e^{14}+e^{23},e^{24}+e^{13},e^{15}+e^{26}&1/12(5,5,8,8,13,13,18)&\{\emptyset,12347,1256,135,1467,236,2457,34567\}\\
731:22b&0,0,0,0,-e^{14}+e^{23},e^{13}+e^{24},e^{26}+e^{15}&1/12(5,5,8,8,13,13,18)&\{\emptyset,12347,1256,135,1467,236,2457,34567\}\\
731:23&0,0,0,0,e^{12},e^{24}+e^{13},e^{26}+e^{35}+e^{14}&2/19(6,3,5,8,9,11,14)&\{12347,126,135,14567\}\\
731:24a&0,0,0,0,e^{24}+e^{13},-e^{12}+e^{34},e^{15}+e^{23}+e^{46}&4/11(1,2,2,1,3,3,4)&\{12347,125,23567,246\}\\
731:24b&0,0,0,0,-e^{13}+e^{24},e^{34}-e^{12},e^{46}+e^{15}+e^{23}&4/11(1,2,2,1,3,3,4)&\{12347,125,23567,246\}\\
73:2&0,0,0,0,e^{12},e^{13},e^{14}&2/5(1,2,2,2,3,3,3)&
\begin{array}{l}
\{\emptyset,1234,1237,1246,1267,1345,1357,1456,\\
1567,234567,2356,2457,25,3467,36,47\}
\end{array}\\
73:3&0,0,0,0,e^{12},e^{13},e^{24}&1/7(4,4,5,5,8,9,9)&
\begin{array}{l}
\{\emptyset,1234,1237,1246,1267,13457,135,14567,\\
156,23456,23567,245,257,3467,36,47\}
\end{array}\\
73:4&0,0,0,0,e^{12},e^{13},e^{23}+e^{14}&1/6(3,4,4,5,7,7,8)&\{\emptyset,1234,1267,1357,1456,2356,2457,3467\}\\
73:5&0,0,0,0,e^{12},e^{34},e^{24}+e^{13}&5/8(1,1,1,1,2,2,2)&\{\emptyset,1234,127,1356,14567,23567,2456,347\}\\
73:6a&0,0,0,0,e^{12},e^{14}+e^{23},e^{13}+e^{24}&1/7(4,4,5,5,8,9,9)&\{\emptyset,1234,1267,1356,1457,2357,2456,3467\}\\
73:6b&0,0,0,0,e^{12},-e^{14}+e^{23},e^{13}+e^{24}&1/7(4,4,5,5,8,9,9)&\{\emptyset,1234,1267,1356,1457,2357,2456,3467\}\\
73:7a&0,0,0,0,e^{23}+e^{14},e^{24}+e^{13},e^{34}+e^{12}&5/8(1,1,1,1,2,2,2)&\{\emptyset,1234,1256,1357,1467,2367,2457,3456\}\\
73:7b&0,0,0,0,e^{23}+e^{14},-e^{13}+e^{24},e^{34}+e^{12}&5/8(1,1,1,1,2,2,2)&\{\emptyset,1234,1256,1357,1467,2367,2457,3456\}\\
721:4&0,0,0,0,0,e^{12},e^{13}+e^{45}+e^{26}&2/13(4,3,6,5,5,7,10)&\{\emptyset,1247,1257,1467,1567,2456,26,45\}\\
72:5&0,0,0,0,0,e^{12},e^{13}+e^{45}&1/7(4,5,6,5,5,9,10)&
\begin{array}{l}
\{\emptyset,123,12345,1247,1257,13456,136,1467,\\
1567,23467,23567,2456,26,347,357,45\}
\end{array}\\
72:6&0,0,0,0,0,e^{24}+e^{13},e^{12}+e^{35}&1/8(6,5,5,6,6,11,11)&
\{\emptyset,12345,126,137,14567,2367,2457,3456\}\\
71:3&0,0,0,0,0,0,e^{56}+e^{34}+e^{12}&4/5(1,1,1,1,1,1,2)&
\begin{array}{l}
\{\emptyset,12,1234,123456,1256,1357,1367,1457,\\
1467,2357,2367,2457,2467,34,3456,56\}
\end{array}
\end{longtable}
}
\end{landscape}

\FloatBarrier
\section{Nilsolitons of dimension $8$ and $9$}
In higher dimensions, the methods of this paper do not seem sufficient to obtain a complete classification, since the number of polynomial systems that need to be solved increases dramatically. For instance, in dimension $8$ there are $119$ nice Lie algebras and $37$ families of nice Lie algebras such that the corank is greater than one, meaning that \ref{enum:condP} consists of two or more equations. However, the method works if we put restrictions on the corank. Again, we use the classification of~\cite{ContiRossi:Construction} with minor differences (see Remark~\ref{remark:accorgimenti}).
\begin{proposition}
\label{prop:8}
The irreducible nice nilpotent Lie algebras of dimension $8$ with $\dim \coker M_\Delta\leq 1$ that admit a diagonal nilsoliton metric of type Nil4 are listed in Table~\ref{table:Nilsoliton8Families} and Table~\tableottocoker\ (see ancillary files); for each Lie algebra, the last column gives the set of signatures of diagonal metrics satisfying~\eqref{eqn:normalizednil4}.
\end{proposition}
\begin{proof}
Because of the Nil4 condition, we restrict to nice Lie algebras with nonzero Nikolaevsky derivation.

We apply Corollary~\ref{cor:KHLP} case by case. Lie algebras that do not depend on a parameter can be handled by the program~\cite{demonblast}, since Condition~\ref{enum:condP} is a single polynomial equation.

In addition, there are eight one-parameter families to consider. For \texttt{852:26}
there is no metric by \ref{enum:condH}. For
\texttt{852:23} and \texttt{842:41}
we have the metrics guaranteed by Corollary~\ref{cor:positivesignature}, and no others because of \ref{enum:condL}.

Then we have five one-parameter families for which Condition~\ref{enum:condP} is a
polynomial condition depending on one parameter.
For \texttt{8531:46},
we compute
\[X=\left(\frac{3}{19},-\frac{1}{19}+x_7,\frac{7}{19}-x_7,-\frac{1}{19}+x_7,\frac{6}{19}-x_7,\frac{6}{19}-x_7,x_7\right).\]
Then
\[x_7\left(x_7-\frac1{19}\right)^2=a^2 \left(\frac{6}{19}-x_7\right)^2\left(\frac{7}{19}-x_7\right).\]
We set $A=a^2$ and rewrite the equation as
\[p=-\frac{252 A}{6859} + \left(\frac{1}{361} + \frac{120 A}{361}\right) x  -\left(\frac{2}{19} +A\right) x^2 + (1 + A) x^3=0.\]
This equation is invariant under
\[A\mapsto \frac1A, \quad x_7\mapsto \frac7{19}-x_7.\]
Therefore, we may assume $A\geq 1$. The discriminant of $p$ is negative for
\[1\leq A \leq \alpha=\frac1{48}(34343 + 323 \sqrt{11305})\sim 1431.0.\]
Solving for two fixed values of $A$, say $A=1$ and $A=1431$, and using continuity, we see that a solution with all entries of $X$ positive exists for every $A\geq 1$, and in addition one or two solutions with $\frac{6}{19}-x_7<0$ and the other entries positive when $A\geq \alpha$.

For $A<1$, we see that for $A\leq \frac1{\alpha}$ we have a solution with $-\frac1{19}+x_7<0$ and the other entries positive.

Therefore, we obtain the signatures
\[\begin{cases}\{\emptyset,134578 ,13478,1458,148, 357,37,  5\} & a^2\leq \frac1\alpha\\
\{\emptyset,13478,1458 , 357\} & \frac1\alpha<a^2< \alpha\\
\{\emptyset,123678, 12568 ,13478,1458,234567, 246,357\} & a^2\geq\alpha\end{cases}.\]

For the two one-parameter families \texttt{8531:58a} and \texttt{8531:58b},
we compute
\[X=\left( \frac{3}{14},x_7,\frac{5}{14}-x_7,-\frac{1}{14}+x_7,\frac{3}{14}-x_7,\frac{3}{14}-x_7,x_7,\frac{3}{14}\right),\]
giving
\[ -196 \frac{ x_7^{2} {(-1+14 x_7)}}{ {(-3+14 x_7)}^{2} {(-5+14 x_7)}}=a^{2}.\]
This equation has solutions with $x_5,x_6<0$ and other entries positive for $a^2\geq 1/8 (191 + 23 \sqrt{69})$ ($\{123678,12568,234567,246\}$); in addition, it has a solution with all entries positive for every value of $a$
($\{\emptyset,13478,1458,357\}$).

For \texttt{842:74}
we compute
\[X= \left(\frac{1}{19}+x_7,\frac{6}{19}-x_7,x_7,\frac{5}{19}-x_7,\frac{5}{19}-x_7,\frac{3}{19},x_7\right),\]
and consequently
\[ -361 \frac{ {(1+19 x_7)} x_7^{2}}{ {(-6+19 x_7)} {(-5+19 x_7)}^{2}}=a^{2}.\]
For any value of $a$ we have solutions with all-positive entries ($\mathbf{S}_0=\{\emptyset,12378,1358,257\}$). In addition, for $a^2\leq (34343 - 323 \sqrt{11305})/48$ we  have $x_3, x_7<0$ and $x_4,x_5>0$, hence $\mathbf{S}=\mathbf{S}_0\cup\{12348,134578,245,47\}$); for $a^2\geq (34343 + 323 \sqrt{11305})/48$ we have $x_3, x_7>0$ and $x_4,x_5<0$, so $\mathbf{S}=\mathbf{S}_0\cup\{124678,14568,234567,346\}$.

For \texttt{842:88}
we have
\[X=\left(x_6+\frac2{11},\frac3{11}-x_6,\frac1{11}-x_6,\frac3{11},x_6,x_6,\frac1{11}-x_6,\frac3{11}\right)\]
giving
\[x_6^2\left(x_6+\frac2{11}\right)=a^2\left(\frac3{11}-x_6\right)\left(\frac1{11}-x_6\right)^2\]
We have a solution with all positive $x_i$ for all values of $a$ ($\mathbf{S}_0=\{\emptyset,12378\}$).
For $a^2\leq(1523 - 77\sqrt{385})/192$ we also have a solution with $x_5,x_6<0$ and $x_3,x_7>0$, giving $\mathbf{S}=\mathbf{S}_0\cup\{134578,245\}$; for $a^2\geq (1523 + 77\sqrt{385})/192$ we also have $x_5,x_6>0$ and $x_3,x_7<0$, so that $\mathbf{S}=\mathbf{S}_0\cup\{124678,346\}$.
\end{proof}

\begin{footnotesize}
{\setlength{\tabcolsep}{2pt}
\begin{longtable}[c]{>{\ttfamily}c C C }
\caption{Families of nice nilpotent Lie algebras of dimension $8$ with $\dim \coker M_\Delta\leq 1$ that admit a diagonal nilsoliton metric of type Nil4\label{table:Nilsoliton8Families}}\\
\toprule
\textnormal{Name} & \g  &  \mathbf{S} \\
\midrule
\endfirsthead
\multicolumn{3}{c}{\tablename\ \thetable\ -- \textit{Continued from previous page}} \\
\toprule
\textnormal{Name} & \g  &  \mathbf{S} \\
\midrule
\endhead
\bottomrule\\[-7pt]
\multicolumn{3}{c}{\tablename\ \thetable\ -- \textit{Continued to next page}} \\
\endfoot
\bottomrule\\[-7pt]
\endlastfoot
8531:46&\begin{array}{c}
0,0,0,e^{12}, a e^{13},\\
e^{14},e^{15}+e^{23},e^{24}+e^{16}
\end{array}&
\begin{array}{rl}
a^2\leq\frac1{\alpha} & \{\emptyset,134578 ,13478,1458,148, 357,37,  5\} \\
 \frac1\alpha<a^2< \alpha& \{\emptyset ,13478,1458, 357\} \\
a^2\geq\alpha&\{\emptyset , 123678, 12568,13478,1458, 234567,246,357\}
\end{array}\\
&&\alpha=\frac1{48}(34343 + 323 \sqrt{11305})\\\cmidrule(lr){3-3}
8531:58a&\begin{array}{c}
0,0,0,e^{12}, a e^{13},\\
e^{14},e^{15}+e^{23},e^{35}+e^{24}+e^{16}
\end{array}&
\begin{array}{rl}
a^2\leq\frac{1}{8}(191+23\sqrt{69}) & \{\emptyset,13478,1458,357\}\\
a^2\geq\frac{1}{8}(191+23\sqrt{69}) & \{\emptyset,123678,12568,13478,1458,234567,246,357\}
\end{array}\\\cmidrule(lr){3-3}
8531:58b&\begin{array}{c}
0,0,0,- e^{12}, a e^{13},\\
e^{14},e^{15}+e^{23},e^{35}+e^{24}+e^{16}
\end{array}&
\begin{array}{rl}
a^2\leq\frac{1}{8}(191+23\sqrt{69}) & \{\emptyset,13478,1458,357\}\\
a^2\geq\frac{1}{8}(191+23\sqrt{69}) & \{\emptyset,13478,1458,357,246,234567,12568,123678\}
\end{array}\\\cmidrule(lr){3-3}
852:23&\begin{array}{c}
0,0,0,e^{12}, a e^{13},\\
e^{23},e^{35}+e^{14},e^{24}+e^{36}
\end{array}&\{\emptyset,12378,125678,356\}\\\cmidrule(lr){3-3}
842:41&\begin{array}{c}
0,0,0,0, a e^{12},\\
e^{13},e^{24}+e^{15},e^{34}+e^{16}
\end{array}&\{\emptyset,12378,1267,1358,156,235678,257,368\}\\\cmidrule(lr){3-3}
842:74&\begin{array}{c}
0,0,0,0, a e^{12},\\
e^{13},e^{24}+e^{15},e^{34}+e^{25}+e^{16}
\end{array}&
\begin{array}{rl}
a^2\leq \alpha_1 & \{\emptyset,12348,12378,134578,1358,257,245,47\}\\
\alpha_1< a^2 < \alpha_2 & \{\emptyset,12378,1358,257\}\\
\alpha_2 \leq a^2 &\{\emptyset,12378,124678,1358,14568,234567,257,346\}\\
\end{array}\\
&&\alpha_1=\frac{1}{48}(34343 - 323 \sqrt{11305}),\quad \alpha_2=\frac{1}{48}(34343 + 323 \sqrt{11305})\\\cmidrule(lr){3-3}
842:88&\begin{array}{c}
0,0,0,0, a e^{12},\\
e^{13},e^{34}+e^{25}+e^{16},e^{24}+e^{36}+e^{15}
\end{array}&
\begin{array}{rl}
a^2\leq \beta_1 & \{\emptyset,12378,134578,245\}\\
\beta_1< a^2< \beta_2 & \{\emptyset,12378\}\\
\beta_2\leq a^2 & \{\emptyset,12378,124678,346\}\\
\end{array}\\
&&\beta_1=\frac{1}{192}(1523 - 77\sqrt{385}),\quad\beta_2=\frac{1}{192}(1523 + 77\sqrt{385})\\
\end{longtable}
}
\end{footnotesize}

For dimension $9$, even the corank one case appears to be intractable, since there are $72$ families of nice Lie algebras with corank one, for which \ref{enum:condP} is a parametric equation. For corank zero, however, we easily obtain:
\begin{proposition}
\label{proposition:9}
The irreducible nice nilpotent Lie algebras of dimension $9$ with $\dim \coker M_\Delta=0$ that admit a diagonal nilsoliton metric of type Nil4 are listed in Table~\tablenovecoker; for each Lie algebra, the last column gives the set of signatures of diagonal metrics satisfying~\eqref{eqn:normalizednil4}.
\end{proposition}

Things become easier if we restrict to Riemannian signature: indeed, a straightforward application of Corollary~\ref{cor:positivesignature} gives:
\begin{theorem}
The irreducible nice nilpotent Lie algebras of dimension $8$ with $\dim \coker M_\Delta>1$ (respectively, dimension $9$ and $\dim \coker M_\Delta>0$) that admit a Riemannian nilsoliton metric are listed in Table~\tableottoriemannian\ (resp. Table~\tablenoveriemannian); for each Lie algebra, the column $\mathbf{S}_0$ gives the set of signatures of nilsoliton metrics obtained by applying an element of $\ker M_{\Delta,2}$.
\end{theorem}
\begin{proof}
By \cite[Theorem~3]{Nikolayevsky}, a nice nilpotent Lie algebra has a Riemannian nilsoliton metric if and only if it has a diagonal Riemannian nilsoliton metric. In addition,~\cite{Nikolayevsky} shows that determining whether a Riemannian diagonal nilsoliton metric exists amounts to solving a system of linear equalities and inequalities
(see also Corollary~\ref{cor:positivesignature}). Case-by-case calculations with~\cite{demonblast} yield the tables Table~\tableottoriemannian\ and Table~\tablenoveriemannian.

The non-Riemannian signatures are computed applying Corollary~\ref{cor:positivesignature}.
\end{proof}

Notice that together with Proposition~\ref{prop:8} and Proposition~\ref{proposition:9}, this gives a full classification of Riemannian nice nilsolitons up to dimension $9$.

\begin{remark}
For each nice Lie algebra, regardless of the corank of $M_\Delta$, one can determine the set of signatures
\begin{equation}
\label{eqn:insiemeinnominato}
\left\{\delta\mid \exists X \text{ satisfying \ref{enum:condK} with $\lambda=-\frac12$, \ref{enum:condH} and } M_{\Delta,2}\delta=\logsign X\right\}.
\end{equation}
Such a signature $\delta$ is in $\mathbf{S}$ if and only if $X$ can be chosen to satisfy \ref{enum:condP}. Thus, if every element of~\eqref{eqn:insiemeinnominato} is in $\mathbf{S}_0$, we deduce that $\mathbf{S}_0=\mathbf{S}$ without having to solve any polynomial equation.

We implemented this strategy in~\cite{demonblast}, using the Fourier-Motzkin algorithm to determine the set~\eqref{eqn:insiemeinnominato}. This allows us to conclude that for some nice Lie algebras $\mathbf{S}_0$ exhausts the set of signatures of diagonal nilsoliton metrics satisfying~\eqref{eqn:normalizednil4}. These Lie algebras are flagged with a check mark in the last column of Table~\tableottoriemannian\ and Table~\tablenoveriemannian. We emphasize that the lack of a check mark does not imply that $\mathbf{S}_0$ is strictly contained in $\mathbf{S}$.
\end{remark}

\begin{remark}
 \label{remark:comparison}
If one only considers Riemannian signature, our tables can be compared to the existing classifications as follows. In dimension $7$, our results are compatible with those of~\cite{FernandezCulma}, though we only consider nice Lie algebras here. In dimension $8$, we recover the classification of~\cite{KadiogluPayne}, but our result is more general since we do not require  that the root matrix be surjective and the eigenvalues of the Nikolayevsky derivation distinct. For filiform Lie algebras, we recover the results of~\cite{Arroyo:Filiform} concerning the existence of a Riemannian nilsoliton metric on filiform Lie algebras of dimension $8$. We remark that all the Lie algebras appearing in~\cite{Arroyo:Filiform} are nice: indeed, they already appear in a nice basis except $\lie{m}_1(8)$, which is isomorphic to \texttt{8654321:5} under a simple change of basis. Table~\ref{table:FiliformDim8} makes explicit the correspondence between our nice Lie algebras and the Lie algebras of~\cite{Arroyo:Filiform}.
\end{remark}

{\setlength{\tabcolsep}{2pt}
\begin{table}[thp]
\centering
\caption{Correspondence with~\cite{Arroyo:Filiform} for filiform $8$-dimensional Lie algebras \label{table:FiliformDim8}}
\begin{tabular}{>{\ttfamily}c  C  C c}
\toprule
\textnormal{Name}& \textnormal{\cite{Arroyo:Filiform}} & \g & Riem. Nil.\\
\midrule
8654321:1 & \lie{m}_0(8) & 0,0,e^{12},e^{13},e^{14},e^{15},e^{16},e^{17}&$\checkmark$\\
8654321:2 & \lie{g}_1(8) & 0,0,e^{12},e^{13},e^{14},e^{15},e^{16},e^{17}+e^{23}&$\checkmark$\\
8654321:3 & \lie{d}_1(8) & 0,0,e^{12},e^{13},e^{14},e^{15},e^{16}+e^{23},e^{24}+e^{17}&$\checkmark$\\
8654321:4 & \lie{a}_{-1}(8) & 0,0,- e^{12},e^{13},e^{14},e^{15},e^{16},e^{25}+e^{34}+e^{17}&  \\
8654321:5 & \lie{m}_1(8) & 0,0,- e^{12},- e^{13},e^{14},e^{15},e^{16},e^{45}+e^{36}+e^{27}&$\checkmark$\\
8654321:6 & \lie{c}_{1,0}(8) & 0,0,e^{12},e^{13},e^{14},e^{15}+e^{23},e^{24}+e^{16},e^{25}+e^{17}&\\
8654321:7 & \lie{a}_0(8) & 0,0,e^{12},e^{13},e^{14},e^{15}+e^{23},e^{24}+e^{16},e^{34}+e^{17}&$\checkmark$\\
8654321:11 & \lie{s}_1(8) & 0,0,- e^{12},- e^{13},e^{14},e^{15},e^{16}+e^{23},e^{36}+e^{27}+e^{45}&$\checkmark$\\
8654321:12 & \lie{m}_2(8) & 0,0,e^{12},e^{13},e^{23}+e^{14},e^{24}+e^{15},e^{25}+e^{16},e^{26}+e^{17}& \\
8654321:14 & \begin{array}{c}
\lie{a}_t(8)\\ t\neq0,-1
\end{array} & \begin{array}{c}
0,0,- e^{12} {(-1+a_3)},e^{13}, e^{14} a_3,\\
e^{15}+e^{23},e^{24}+e^{16},e^{25}+e^{34}+e^{17}
\end{array}&$\checkmark$\\
8654321:15 & \lie{g}_{-2}(8) & 0,0,- e^{12},e^{13},e^{14},2  e^{15},e^{25}+e^{34}+e^{16},e^{35}+e^{26}+e^{17}&\\
8654321:17 & \lie{k}_1(8) & 0,0,- e^{12},e^{13},e^{14},-e^{15}+e^{23},e^{24}+e^{16},e^{36}+e^{27}+e^{45}&$\checkmark$\\
8654321:19 & \lie{g}_0(8) & \begin{array}{c}
0,0,e^{12},e^{13},e^{23}+e^{14},e^{24}+e^{15},\\
e^{25}+e^{34}+2 e^{16},e^{35}+e^{17}
\end{array}&$\checkmark$\\
8654321:20 & \lie{g}_{-1}(8) & \begin{array}{c}
0,0,e^{12},- e^{13},e^{23}+e^{14},e^{24}-e^{15},\\
e^{34}+e^{16},e^{35}+e^{26}+e^{17}
\end{array}&$\checkmark$\\
8654321:24 & \begin{array}{c}
\lie{g}_{\alpha}(8)\\ \alpha\neq0,-1,-2
\end{array}& \begin{array}{c}
0,0, e^{12} {(1-a_4)}, e^{13} a_4,e^{23}+e^{14},\\
e^{24}+ e^{15} a_4, e^{16} {(2-a_4)}+e^{25}+e^{34},e^{35}+e^{26}+e^{17}
\end{array}&$\checkmark$\\
8654321:25 & \lie{b}(8) & \begin{array}{cc}
0,0,- e^{12},- e^{13},-\frac{1}{2} e^{23}+\frac{3}{2} e^{14},\\
\frac{1}{2} e^{24}+e^{15},e^{25}+e^{34}+e^{16},e^{36}+e^{27}+e^{45}
\end{array}&$\checkmark$\\
\bottomrule
\end{tabular}
\end{table}
}

\FloatBarrier

\bibliographystyle{plain}
\bibliography{NiceNilsoliton}

\begin{thebibliography}{10}

\bibitem{Arroyo:Filiform}
R.~M. Arroyo.
\newblock Filiform nilsolitons of dimension 8.
\newblock {\em Rocky Mountain J. Math.}, 41(4):1025--1043, 2011.

\bibitem{demonblast}
D.~Conti.
\newblock {DEMONbLAST}, a program to compute {D}iagonal {E}instein {M}etrics
  {O}n {N}ice {L}ie {A}lgebras of {S}urjective {T}ype.
\newblock \url{https://github.com/diego-conti/DEMONbLAST}.

\bibitem{ContiRossi:IndefiniteNilsolitons}
D.~Conti and F.~A. Rossi.
\newblock Indefinite nilsolitons and {E}instein solvmanifolds.
\newblock arXiv:2105.09209.

\bibitem{ContiRossi:Construction}
D.~Conti and F.~A. Rossi.
\newblock {C}onstruction of nice nilpotent {L}ie groups.
\newblock {\em Journal of Algebra}, 525:311 -- 340, 2019.

\bibitem{ContiRossi:RicciFlat}
D.~Conti and F.~A. Rossi.
\newblock {R}icci-flat and {E}instein pseudoriemannian nilmanifolds.
\newblock {\em Complex Manifolds}, 6(1):170--193, 2019.

\bibitem{ContiRossi:EinsteinNice}
D.~Conti and F.~A. Rossi.
\newblock Indefinite {E}instein metrics on nice {L}ie groups.
\newblock {\em {F}orum {M}athematicum}, 32(6):1599--1619, 2020.

\bibitem{FernandezCulma}
E.~A. Fern\'{a}ndez-Culma.
\newblock Classification of nilsoliton metrics in dimension seven.
\newblock {\em J. Geom. Phys.}, 86:164--179, 2014.

\bibitem{Heber:noncompact}
J.~Heber.
\newblock Noncompact homogeneous {E}instein spaces.
\newblock {\em Invent. Math.}, 133(2):279--352, 1998.

\bibitem{Helleland:WickRotations}
C.~Helleland.
\newblock {W}ick-rotations of pseudo-{R}iemannian {L}ie groups.
\newblock {\em Journal of Geometry and Physics}, 158:103902, 2020.

\bibitem{Jablonski:Orbits}
M.~Jablonski.
\newblock Distinguished orbits of reductive groups.
\newblock {\em Rocky Mountain J. Math.}, 42(5):1521--1549, 2012.

\bibitem{KadiogluPayne}
H.~Kadioglu and T.~L. Payne.
\newblock Computational methods for nilsoliton metric {L}ie algebras {I}.
\newblock {\em J. Symbolic Comput.}, 50:350--373, 2013.

\bibitem{Lauret:RicciSoliton}
J.~Lauret.
\newblock Ricci soliton homogeneous nilmanifolds.
\newblock {\em Math. Ann.}, 319(4):715--733, 2001.

\bibitem{Lauret:Finding}
J.~Lauret.
\newblock Finding {E}instein solvmanifolds by a variational method.
\newblock {\em Math. Z.}, 241(1):83--99, 2002.

\bibitem{Lauret:Einstein_solvmanifolds}
J.~Lauret.
\newblock Einstein solvmanifolds are standard.
\newblock {\em Ann. of Math. (2)}, 172(3):1859--1877, 2010.

\bibitem{LauretWill:Einstein}
J.~Lauret and C.~Will.
\newblock Einstein solvmanifolds: existence and non-existence questions.
\newblock {\em Math. Ann.}, 350(1):199--225, 2011.

\bibitem{LauretWill:OnTheDiagonalization}
J.~Lauret and C.~Will.
\newblock On the diagonalization of the {R}icci flow on {L}ie groups.
\newblock {\em Proc. Amer. Math. Soc.}, 141(10):3651--3663, 2013.

\bibitem{Nikolayevsky:EinsteinDerivation}
Y.~Nikolayevsky.
\newblock Einstein solvmanifolds with a simple {E}instein derivation.
\newblock {\em Geom. Dedicata}, 135:87--102, 2008.

\bibitem{Nikolayevsky}
Y.~Nikolayevsky.
\newblock Einstein solvmanifolds and the pre-{E}instein derivation.
\newblock {\em Trans. Amer. Math. Soc.}, 363(8):3935--3958, 2011.

\bibitem{Payne:TheExistence}
T.~L. Payne.
\newblock The existence of soliton metrics for nilpotent {L}ie groups.
\newblock {\em Geom. Dedicata}, 145:71--88, 2010.

\bibitem{Payne:Applications}
T.~L. Payne.
\newblock Applications of index sets and {N}ikolayevsky derivations to positive
  rank nilpotent {L}ie algebras.
\newblock {\em J. Lie Theory}, 24(1):1--27, 2014.

\bibitem{Will:RankOne}
C.~Will.
\newblock Rank-one {E}instein solvmanifolds of dimension 7.
\newblock {\em Differential Geom. Appl.}, 19(3):307--318, 2003.

\bibitem{Yan:Pseudo-RiemannianEinsteinhomogeneous}
Z.~Yan.
\newblock {P}seudo-{R}iemannian {E}instein metrics on noncompact homogeneous
  spaces.
\newblock {\em J.Geom.}, 111, 2020.

\end{thebibliography}

\small\noindent Dipartimento di Matematica e Applicazioni, Universit\`a di Milano Bicocca, via Cozzi 55, 20125 Milano, Italy.\\
\texttt{diego.conti@unimib.it}\\
\texttt{federico.rossi@unimib.it}

\end{document}